\documentclass{amsart}

\usepackage{amsmath,amsthm,amssymb,amsfonts}
\usepackage[a4paper]{geometry}
\usepackage{a4wide}
\usepackage{graphicx}              
\usepackage{mathpazo}
\usepackage[T1]{fontenc}
\usepackage[utf8]{inputenc}
\usepackage{enumitem}
\usepackage{hyperref}
\usepackage{xcolor}
\hypersetup{
    colorlinks,
    linkcolor={red!80!black},
    citecolor={blue!80!black},
    urlcolor={blue!80!black}
}
\usepackage{tikz}
\usepackage{ mathrsfs } 
\usetikzlibrary{matrix,arrows, cd, calc}  
\usepackage[margin=10pt,font=normalsize,labelfont=bf, labelsep=endash]{caption}

\makeatletter
\def\thm@space@setup{\thm@preskip=4pt
\thm@postskip=2pt}
\makeatother

\newtheorem{theorem}{Theorem}[section]

\newtheorem{corollary}[theorem]{Corollary}

\newtheorem{proposition}[theorem]{Proposition}

\newtheorem{lemma}[theorem]{Lemma}

\theoremstyle{definition}
\newtheorem{definition}[theorem]{Definition}
\newtheorem{assumption}[theorem]{Assumption}
\newtheorem{example}[theorem]{Example}

\newtheorem{remark}[theorem]{Remark}

\newcommand{\NN}{{\mathbb N}}

\newcommand{\cZ}{{\mathcal Z}}

\newcommand{\cX}{{\mathcal X}}
\newcommand{\cW}{{\mathcal W}}

\newcommand{\cU}{{\mathcal U}}

\newcommand{\cO}{{\mathcal O}}

\newcommand{\cM}{{\mathcal M}}

\newcommand{\cI}{{\mathcal I}}

\newcommand{\cC}{{\mathcal C}}
\newcommand{\cB}{{\mathcal B}}
\newcommand{\cA}{{\mathcal A}}
\newcommand{\cJ}{{\mathcal J}}

\newcommand{\goodquotient}{\mathbin{
  \mathchoice{\left/\mkern-6mu\right/}
    {/\mkern-5mu/}
    {/\mkern-5mu/}
    {/\mkern-5mu/}}}

\newcommand{\Hom}{\mathrm{Hom}}
\newcommand{\Spec}{\mathrm{Spec}\,}

\newcommand{\Sym}{\textrm{Sym}}
\newcommand{\Mor}{\textrm{Mor}}

\newcommand{\Set}{\mathbf{Set}}

\newcommand{\Sch}{\mathbf{Sch}}

\newcommand{\Group}{\mathbf{G}}%
\newcommand{\Groupop}{\mathbf{G}^{\mathrm{op}}}%
\newcommand{\Gprod}{\Group\times \Groupop}%
\newcommand{\Groupconn}{\mathbf{G}^{\circ}}%
\newcommand{\Nroup}{\mathbf{N}}%
\newcommand{\Nred}{\Nroup_{\mathrm{red}}^{\circ}}%
\newcommand{\Gbar}{\overline{\mathbf{G}}}%
\newcommand{\Nbar}{\overline{\mathbf{N}}}%
\newcommand{\Fbar}{\overline{\mathbf{F}}}%

\newcommand{\Gmult}{\mathbb{G}_m}%
\newcommand{\kk}{k}%
\newcommand{\kkbar}{\overline{k}}%
\newcommand{\varX}{X}%
\newcommand{\kSch}{\Sch_{\kk}}%
\newcommand{\Dfunctor}[1]{\mathcal{D}_{#1}}%
\newcommand{\DX}{\Dfunctor{\varX}}%
\newcommand{\Xplus}{\varX^+}%
\newcommand{\Yplus}{Y^+}%
\newcommand{\fplus}{f^+}%
\newcommand{\Ffunctor}[1]{{#1}^{\Group}}%
\newcommand{\Hfunctor}[1]{\widehat{\mathcal{D}}_{#1}}%
\newcommand{\ione}[1]{i_{#1}}%
\newcommand{\ioneX}{\ione{\varX}}
\newcommand{\iinfty}[1]{\pi_{#1}}%
\newcommand{\iinftyX}{\iinfty{\varX}}%
\newcommand{\isection}[1]{s_{#1}}%
\newcommand{\isectionX}{\isection{\varX}}%
\newcommand{\BBname}{Bia{\l}ynicki-Birula}%
\DeclareMathOperator{\Maps}{Maps}
\DeclareMathOperator{\Irr}{Irr}
\DeclareMathOperator{\colim}{colim}

\DeclareMathOperator{\charr}{char}%
\DeclareMathOperator{\id}{id}%
\DeclareMathOperator{\Hilb}{Hilb}%
\newcommand{\into}{\hookrightarrow}%
\newcommand{\onto}{\twoheadrightarrow}%
\newcommand{\sigmabar}{\overline{\sigma}}%

\newcommand{\GroupGlobalSects}{H^0(\Group, \cO_{\Group})}%
\newcommand{\mubar}{\overline{\mu}}%

\begin{document}

\title{Białynicki-Birula decomposition for reductive groups}
\author{Joachim Jelisiejew}
\thanks{\emph{Corresponding author.} Institute of Mathematics, Polish Academy of
    Sciences ({\'S}niadeckich 8, 00-656 Warsaw), \url{jjelisiejew@impan.pl},
    partially supported by Polish National Science
Center, project 2017/26/D/ST1/00755.}

    \author{{\L{}}ukasz Sienkiewicz}
    \thanks{Faculty of Mathematics, Informatics and Mechanics, University of
        Warsaw (Banacha 2, 02-097 Warsaw),
        \url{lusiek@mimuw.edu.pl}, supported by Polish National Science Center project 2013/08/A/ST1/00804.}


\dedicatory{To professor Piotr Grzeszczuk, for his 60th birthday}
\begin{abstract}
    We generalize the \BBname{} decomposition from actions of $\Gmult$ on smooth
    varieties to actions of linearly reductive group $\Group$ on finite type schemes and
    algebraic spaces. We also provide a relative version and briefly discuss
    the case of algebraic stacks.

    We define the \BBname{} decomposition
    functorially: for a fixed $\Group$-scheme $X$ and a monoid
    $\Gbar$ which partially compactifies $\Group$, the BB decomposition parameterizes
    $\Group$-schemes over $X$ for which the $\Group$-action extends to the
    $\Gbar$-action. The freedom of choice of $\Gbar$ makes the theory richer
    than the $\Gmult$-case.
\end{abstract}
\maketitle

\tableofcontents
\section{Introduction}

    The classical \BBname{}
    decomposition~\cite{BialynickiBirula__decomposition, Carrell_survey_on_BB}
    is a smooth variety $\Xplus\to \varX$ obtained from a smooth $\Gmult$-variety
    $\varX$.
    Drinfeld~\cite{Drinfeld} observed that it admits a functorial
    description
    \begin{equation}\label{eq:definitionFunctorial:new}
        \Xplus(S) = \left\{ \varphi\colon \mathbb{A}^1 \times S \to \varX\ |\
        \varphi \mbox{ is $\Gmult$-equivariant} \right\},
    \end{equation}
    where $\Gmult$ acts on $\mathbb{A}^1$ naturally and on $S$ trivially.
    The morphism $\Xplus\to \varX$ comes from restriction $\varphi\mapsto
    (\varphi_{|1 \times S} \colon S\to \varX)$. Reading this backwards, an
    $S$-point of $\Xplus$ is an $S$-point of $\varX$ together with its limit
    under the $\Gmult$-action.  One of the main results of~\cite{Drinfeld} is that a
    scheme $\Xplus$ given by~\eqref{eq:definitionFunctorial:new} exists for
    all (not necessarily smooth) schemes $\varX$ of finite type.
    See~\cite{Haines_Richarz, Richarz, vanDerBergh} for some of the applications of this
    result.

    The space $\Xplus$ defined in~\eqref{eq:definitionFunctorial:new} is an important
    tool for analysis of highly singular spaces, such as moduli spaces. For
    example, $\Xplus$ and the morphism $\Xplus\to \varX$ is
    used in~\cite{Jelisiejew__Elementary} to analyse $\varX = \Hilb_d
    \mathbb{A}^n$. In this setting, the space $\Xplus$ has a serious drawback: its
    definition takes into account only a fixed $\Gmult$-action on $\varX$,
    while the automorphism group of $\varX$ is usually larger (for
    $\Hilb_d \mathbb{A}^n$ we have for example the action of $\mathrm{GL}_n$).

    The main aim of this paper is to generalize the results of~\cite{Drinfeld,
    Jelisiejew__Elementary} by replacing $\Gmult$ with an arbitrary linearly
    reductive affine group $\Group$. The functorial
    description~\eqref{eq:definitionFunctorial:new} readily generalizes once
    we understand what should be put in place of $\mathbb{A}^1$. It turns out that
    a suitable replacement is a \emph{linearly reductive monoid} $\Gbar$, i.e., an affine
    variety with multiplication $\Gbar \times \Gbar\to \Gbar$ and a unit, such
    that $\Group \subset \Gbar$ is the dense submonoid consisting of invertible
    elements.

    For $\Group= \Gmult$ there are only two non-trivial linearly reductive
    monoids $\Gbar$: the compactification of $\Gmult$ at zero or
    infinity (both are isomorphic to $\mathbb{A}^1$ but have different
    $\Group$-actions). More generally, for $\Group = \Gmult^n$
    the possible $\Gbar$'s are exactly $n$-dimensional affine toric varieties.
    Another important example in characteristic zero is $\Group =
    \mathbb{GL}_n$ and $\Gbar = \mathbb{M}_{n\times n}$.
	The theory of linearly reductive monoids is developed, e.g.,
	in~\cite{Putcha__Linear_algebraic_Monoids, Renner__Linear_algebraic_monoids,
    Vinberg__on_reductive_algebraic_semigroups, Brion__On_algebraic_semigroups_and_monoids}.
    In general, there are significantly more $\Gbar$'s than just toric
    varieties, see Example~\ref{ex:compactificationsOfGLtwo}. However,
    there are no nontrivial $\Gbar$'s for semisimple $\Group$,
    see e.g.~\cite[Exercise~5, p.~41]{Renner__Linear_algebraic_monoids}.

    Now we explicitly define the \BBname{} decomposition.  Let $\kk$ be a
    field and
    $\varX$ be a $\kk$-scheme with a $\Group$-action $\sigma\colon \Group \times
    \varX\to \varX$. The \emph{\BBname{} decomposition} for $(\varX,
    \sigma, \Gbar)$ is the functor $\DX = \Dfunctor{\varX, \Gbar}\colon
    \kSch^{\mathrm{op}}\to \Set$ defined by
    \begin{equation}\label{eq:BBdefinition}
        \Dfunctor{X}(S) = \left\{ \varphi\colon \Gbar \times S\to \varX\ |\
            \mbox{$\varphi$ is $\Group$-equivariant}\right\},
    \end{equation}
    where $\Group$ acts on $\Gbar$ by left multiplication, acts trivially on $S$
    and acts on $\varX$ via $\sigma$.

    Our first result is that under very  weak assumptions the \BBname{}
    decomposition exists: the functor $\DX$ is represented by a scheme locally
    of finite type. In particular we have no smoothness, normality
    etc.~assumptions for $\varX$. Below we \emph{assume that $\Group$ is connected and that $\Gbar$ has a zero}, i.e., there is $0\in
    \Gbar(\kk)$ such that $0\cdot \Gbar = \Gbar \cdot 0 = 0$. For example,
    $\mathbb{M}_{n\times n}$ has a zero, which is the zero matrix; an affine toric
    variety has a zero iff the corresponding semigroup has no nontrivial
    units. If the field $\kk$ is perfect (e.g., $\kk$ has characteristic zero
    or is
    algebraically closed or is finite), then every monoid $\Gbar$ which is
    normal (as variety) has a submonoid $\Nbar$ with zero and with connected
    unit group such that $\Dfunctor{\varX, \Gbar} = \Dfunctor{\varX, \Nbar}$
    for every $\Group$-scheme $\varX$, see
    Theorem~\ref{ref:reductionToMonoidWithZero:thm}. Hence,
    if $\kk$ is perfect, the assumption that $\Group$ is connected and $\Gbar$ has a zero is always
    satisfied up to a replacement of $\Gbar$ by $\Nbar$.

    For each equivariant family $\varphi\colon \Gbar \times S\to \varX$ as
    in~\eqref{eq:definitionFunctorial:new} we have two natural restrictions. If we
    think of $\Dfunctor{X}$ as parameterizing points of $\varX$ with a limit,
    then those restrictions are called:
    \begin{itemize}
        \item \emph{Forgetting about the limit}. We restrict $\varphi$ to
            $\varphi_{|1 \times S}\colon S\to \varX$. The transformation
            $\varphi\mapsto \varphi_{|1 \times S}$ induces a map $\ioneX\colon
            \Dfunctor{X}\to \varX$,
        \item \emph{Restricting to the limit}. We restrict $\varphi$ to $\varphi_{|0\times S}\colon
            S\to \varX$. The family $\varphi_{|0 \times S}$ is equivariant,
            hence the image lies in $\varX^{\Group}$. Thus we obtain a map
            $\iinftyX\colon \Dfunctor{X}\to \varX^{\Group}$.
    \end{itemize}
    In the special case when $\varX$ is smooth and $\Gbar = \mathbb{A}^1$ the map
    $\ioneX$ is the immersion of the \BBname{} cells into $\varX$, whereas
    $\iinftyX$ is the contraction of cells to the fixed locus.

    Figure~\ref{eq:diagramWhatIsGoingOn} illustrates $\ioneX$ and $\iinftyX$. We
    start with an equivariant $\varphi\colon \Gbar \times S \to \varX$,
    as in the middle.  It corresponds to a family $\ioneX(\varphi)\colon S\to
    \varX$ of points of $\varX$, see the right picture, such that we may consider the
    orbits and uniformly take their limits, to obtain a fixed scheme
    $\iinftyX(\varphi)\colon S\to \varX^{\Group}$,
    as in the left picture.
    \begin{figure}[h]
        \[\resizebox{0.95\textwidth}{!}{
                \begin{tikzpicture}
                    \draw plot [smooth cycle] coordinates {(0, 0) (2, 1.4) (6,
                        2) (8,
                    -1) (3.5, -2)};
                    \draw[line width=0.5mm] plot [smooth] coordinates
                    {(3, 1) (3.7, -0.2) (5.5, -1)};
                    \draw[blue] plot [smooth] coordinates {(5.3,
                    1.3) (3.2, 0.7) (1.5, -0.4)};
                    \begin{scope}[shift={(0.2, -0.2)}]
                        \draw[blue] plot [smooth] coordinates {(5.3,
                        1.3) (3.2, 0.7) (1.5, -0.4)};
                    \end{scope};
                    \begin{scope}[shift={(0.4, -0.4)}]
                        \draw[blue] plot [smooth] coordinates {(5.3,
                        1.3) (3.2, 0.7) (1.5, -0.4)};
                    \end{scope};
                    \begin{scope}[shift={(0.7, -0.6)}]
                        \draw[blue] plot [smooth] coordinates {(5.3,
                        1.3) (3.2, 0.7) (1.5, -0.4)};
                    \end{scope};
                    \begin{scope}[shift={(0.9, -0.8)}]
                        \draw[blue] plot [smooth] coordinates {(5.3,
                        1.3) (3.2, 0.7) (1.5, -0.4)};
                    \end{scope};
                    \begin{scope}[shift={(1.1, -1)}]
                        \draw[blue] plot [smooth] coordinates {(5.3,
                        1.3) (3.2, 0.7) (1.5, -0.4)};
                    \end{scope};
                    \begin{scope}[shift={(1.5, -1.1)}]
                        \draw[blue] plot [smooth] coordinates {(5.3,
                        1.3) (3.2, 0.7) (1.5, -0.4)};
                    \end{scope};
                    \begin{scope}[shift={(1.8, -1.2)}]
                        \draw[blue] plot [smooth] coordinates {(5.3,
                        1.3) (3.2, 0.7) (1.5, -0.4)};
                    \end{scope};
                    \begin{scope}[shift={(2, -1.3)}]
                        \draw[blue] plot [smooth] coordinates {(5.3,
                        1.3) (3.2, 0.7) (1.5, -0.4)};
                    \end{scope};
                    \begin{scope}[shift={(2.2, -1.4)}]
                        \draw[blue] plot [smooth] coordinates {(5.3,
                        1.3) (3.2, 0.7) (1.5, -0.4)};
                    \end{scope};

                    \draw[blue, line width=0.7mm] plot [smooth] coordinates
                    {(4.5, 1.4) (5.2, 0.2) (5.9, 0.2) (6.9, -0.3)};

                    \path[->,line width=1pt] (9,0) edge (11,0);
                    \node (v1) at (10, 0.5) {{\Huge $\ione{\varX}$}};
                    \begin{scope}[shift={(12, 0)}]
                        \draw plot [smooth cycle] coordinates {(0, 0) (2, 1.4) (6,
                            2) (8,
                        -1) (3.5, -2)};
                        \draw[line width=0.5mm] plot [smooth] coordinates {(3,
                        1) (3.7, -0.2) (5.5, -1)};
                        \node at (4.5, 1) {{\Huge $\ioneX(\varphi)$}};
                    \end{scope}

                    \path[->,line width=1pt] (-1,0) edge (-3,0);
                    \node (v2) at (-2, 0.5) {{\Huge $\iinfty{\varX}$}};
                    \begin{scope}[shift={(-12, 0)}]
                        \draw plot [smooth cycle] coordinates {(0, 0) (2, 1.4) (6,
                            2) (8,
                        -1) (3.5, -2)};
                        \draw[blue, line width=0.7mm] plot [smooth] coordinates
                        {(4.5, 1.4) (5.2, 0.2) (5.9, 0.2) (6.9, -0.3)};
                        \node[blue] at (6.9, -0.8) {{\Huge $\iinftyX(\varphi)$}};
                    \end{scope}
            \end{tikzpicture}}\]
            \caption{The natural maps associated to $\Dfunctor{X}$.}\label{eq:diagramWhatIsGoingOn}
        \end{figure}

    The monoid $\Gbar$ naturally acts on $\Dfunctor{X}$: on $S$-points the
    action is $(g\cdot \varphi)(h, s) :=
    \varphi(gh, s)$. If we think of $\Dfunctor{X}$ as parameterizing
    $\Group$-orbits  along with the limit, then this action is just the
    compactification of the $\Group$-action on orbits.
    \begin{theorem}[Existence of \BBname{} decompositions]\label{ref:introRepresentability:thm}
        Let $\varX$ be a $\Group$-scheme locally of finite type over
        $\kk$. Then the functor $\Dfunctor{X}$ defined
        by~\eqref{eq:BBdefinition} is represented by a scheme $\Xplus$ locally
        of finite type over $\kk$.
        Moreover, the scheme $\Xplus$ has a natural $\Gbar$-action. The map
        $\iinftyX\colon \Xplus\to \varX^{\Group}$ is \emph{affine} of finite type and
        equivariant.
    \end{theorem}

    For $\varX$ smooth, $\Group = \Gmult$ and $\Gbar$
    equal to $\Gmult$ compactified at zero (resp.~infinity)
    the resulting $\Xplus$ is the positive (resp.~negative)
    classical BB decomposition of $\varX$, as in~\cite{BialynickiBirula__decomposition}.
    While we formulated Theorem~\ref{ref:introRepresentability:thm} for
    schemes, the same theorem holds for algebraic spaces $\varX$, see
    Theorem~\ref{ref:Representability:thm}. In particular, if $\varX$ is an
    algebraic space with $\varX^{\Group}$ a scheme, then $\Xplus$ is a scheme.

    \bigskip{}
    Let us discuss the consequences of
    Theorem~\ref{ref:introRepresentability:thm} to the structure of
    $\Gbar$-schemes. We begin with a surprising observation that the action of $\Gbar$ is very well-behaved
    under minimal assumptions:
    \begin{corollary}\label{ref:GstableNeighbourhoods:cor}
        Let $\varX$ be a $\Gbar$-scheme locally
        of finite type over $\kk$. Then every point of $\varX$ has an
        affine $\Gbar$-stable neighbourhood.
    \end{corollary}
    \begin{proof}
        Note that $\varX=\Xplus$ by Proposition~\ref{ref:isomorphism:prop}.
        By Theorem~\ref{ref:introRepresentability:thm} the projection
        $\iinftyX$ is $\Gbar$-equivariant and affine. The required
        neighbourhood of $x\in \varX$ is the  preimage of any
        affine open neighbourhood of $\iinftyX(x)\in \varX^{\Group}$.
    \end{proof}
    This is striking, since for an arbitrary, even normal, $\Group$-scheme
    $\varX$ the existence of affine $\Group$-stable neighbourhoods is subtle,
    see Brion~\cite{Brion__linearization}. As an example of non-normal scheme: the
    node of the nodal curve $C = \mathbb{P}^1/(0 = \infty)$ with its natural $\Gmult$-action does not
    admit an affine $\Gmult$-stable neighbourhood; here $C^+$ is equal to
    $\mathbb{A}^1$.

    Now we describe $\Xplus$ for an affine $\varX$. Recall that every
    irreducible finite dimensional $\Group$-representation appears in
    $H^0(\Group, \cO_{\Group})$. An irreducible $\Group$-representation $V$ is
    called a \emph{$\Gbar$-representation} if it appears in $H^0(\Gbar,
    \cO_{\Gbar})$. Otherwise it is called an \emph{outsider} representation.
    The $\Group$-action on $V$ extends to a $\Gbar$-action precisely when $V$
    is a $\Gbar$-representation, see Lemma~\ref{ref:mainRepr:lem}.
    \begin{example}
        If $\Group = \Spec \kk[M]$ is a torus and $\Gbar = \Spec \kk[S]$ is
        affine toric, then the $\Gbar$-representations correspond to the
        characters from $S$, while the outsider representations correspond to
        characters from $M
        \setminus S$. More generally, at least when $\Gbar$ is normal, the
        highest weights of $\Gbar$-representations span a cone inside the Weyl
        chamber~\cite[Theorem~5.10]{Renner__Linear_algebraic_monoids}, cf.~Example~\ref{ex:compactificationsOfGLtwo}.
    \end{example}
    \begin{proposition}[affine case, see
        Proposition~\ref{ref:representabilityForAffine:prop}]\label{ref:affineRepresentability:prop}
        If $\varX = \Spec(A)$ is an affine $\Group$-scheme, then $\ioneX\colon\Xplus \to \varX$ is a closed
        immersion, whose ideal is generated by all outsider representations in
        $A$.
    \end{proposition}
    For \emph{smooth} $\varX$ (of course, not necessarily affine) we recover an exact analogue of the classical \BBname{}
    decomposition plus the obtained fiber bundles admit $\Gbar$-actions, and
    in particular they are $\Group$-affine fiber bundles.
    \begin{theorem}\label{ref:ABBdecomposition:thm:intro}
        \def\comp{F}%
        Let $\varX$ be a smooth $\Group$-variety and $\Gbar$ be a linearly reductive monoid with
        zero.  Then the variety $\varX^{\Group}$ is smooth; let $\comp_1,
        \ldots ,\comp_r$ the be its connected components.  Let $\Xplus$ be the
        \BBname{} decomposition of $\varX$. Then
        \begin{enumerate}
            \item\label{item:ABBfirst} The scheme $\Xplus$ is a smooth
                variety and its connected components $\Xplus_1,
                \ldots , \Xplus_r$ are given by $\Xplus_i = \iinftyX^{-1}(\comp_i)$.
                In other words $\Xplus_i \subset \Xplus$ is the locus with limit point in
                $\comp_i$.
            \item\label{item:ABBsecond} For each $i$ the morphism
                $\iinftyX\colon \Xplus_i\to \comp_i$ is an affine
                fiber bundle (with a $\Gbar$-action fiber-wise) and the morphism $\ioneX\colon \Xplus_i \to
                \varX_i$ is a locally closed immersion.
        \end{enumerate}
    \end{theorem}
    A basic application of
    Theorem~\ref{ref:ABBdecomposition:thm:intro} is presented in
    Example~\ref{ex:BBdecompositionForHilbA2}.

    For singular $\varX$, the map $\iinftyX$ induces a bijection between connected
    components of $\Xplus$ and $\varX^{\Group}$. Of course, for singular
    $\varX$ the scheme $\Xplus$ is usually singular. See
    Examples~\ref{ex:actionGoesRight}-\ref{ex:actionGoesWrong}.

        To compare $\varX$ and $\Xplus$ note that the map
        $\iinftyX\colon \Xplus \to \varX^{\Group}$ has a natural section
        $\isectionX\colon \varX^{\Group} \to \Xplus$ which is defined by
        sending $\varphi'\colon S\to \varX^{\Group} \subset \varX$ to
        $\varphi' \circ pr_2\colon \Gbar \times S \to \varX$. Intuitively,
        this section is the embedding of the fixed point locus into $\Xplus$;
        obviously the trivial $\Group$-action on $\varX^{\Group}$ compactifies
        to a trivial $\Gbar$-action.

        Fix a fixed point $x\in \varX^{\Group}(\kk)$ and let $x' :=
        \isectionX(x)\in \Xplus(\kk)$. Applying
        Proposition~\ref{ref:affineRepresentability:prop} to $\cO_{X,
        x}/\mathfrak{m}_x^2$ we see that the
        cotangent map $T_{\varX, x}^{\vee} \to T_{\Xplus, x'}^{\vee}$ is surjective
        and its kernel consists of outsider representations in $T_{\varX,
        x}^{\vee}$. In the special case when there are no such
        representations, the map $\iinftyX$
        is in fact an open immersion as described in the following
        Proposition~\ref{ref:comparison:intro:prop}.
        \begin{proposition}\label{ref:comparison:intro:prop}
            Let $\varX$ be separated and locally of finite type.
            Let $x\in \varX^{\Group}$ be a fixed $\kk$-point such that the
            cotangent space $T_{\varX,
            x}^{\vee}$ has no outsider representations. Then the map $\ioneX\colon \Xplus \to
            \varX$ is an open embedding near $x' = \isectionX(x)\in \Xplus$.
            More precisely, there exists an affine open
            $\Gbar$-stable neighbourhood
            $U$ of $x'$ such that $(\ioneX)_{|U}\colon U\to \varX$ is an open
            embedding. In particular, $x\in \varX$ has an affine open $\Gbar$-stable
            neighbourhood.
            Conversely, if a fixed point $x$ has an affine open $\Gbar$-stable neighbourhood
            in $\varX$, then $T_{\varX, x}^{\vee}$ has no outsider representations.
        \end{proposition}

        Proposition~\ref{ref:comparison:intro:prop} has the
        following consequence, which is of independent interest:
        if $x$ lies in $\varX^{\Group}$ and there exists a linearly reductive monoid $\Gbar$ with zero
        such that $T_{X, x}$ has no outsider representations (for $\Gbar$),
        then there exists a $\Group$-stable affine open neighbourhood of $x$.
        Proposition~\ref{ref:comparison:intro:prop} is also motivated by the
        results of~\cite{Jelisiejew__Elementary}, in particular
        by~\cite[Proposition~1.5]{Jelisiejew__Elementary} and we plan to
        investigate its consequences for Hilbert and Quot schemes.

        \bigskip

        The functorial description of $\Xplus$ implies that all the above and
        similar results have relative counterparts. Let $f\colon X\to Y$ be an
        equivariant morphism of $\Group$ schemes and $\fplus\colon \Xplus\to
        \Yplus$ be the induced morphism of \BBname{} decompositions. We list
        several properties of $f$ which extend to $\fplus$.
        \begin{enumerate}
            \item\label{it:intro:fst} if $f$ is smooth, then also $\fplus$ is smooth,
                Theorem~\ref{ref:relativesmoothness:thm}. In particular, if
                $\Group$ acts trivially on $Y$, then $\fplus\colon \Xplus\to
                \Yplus = Y$ is a family of \BBname{} decompositions as in
                Theorem~\ref{ref:ABBdecomposition:thm:intro}.
            \item if $f$ is \'etale, then we have cartesian squares
                \begin{equation}
                    \begin{tikzcd}
                        \Xplus \arrow[r, "\iinfty{X}"]\arrow[d, "\fplus"] & X^{\Group} \arrow[r]\arrow[d] & X\arrow[d, "f"]\\
                       \Yplus \arrow[r, "\iinfty{Y}"] & Y^{\Group} \arrow[r] &
                       Y
                    \end{tikzcd}
                \end{equation}
                by Theorem~\ref{ref:etale:thm}. In particular, $\fplus$ is a
                base change of $f$, hence is \'etale.
            \item as a special case, if $f$ is an open immersion, then $\fplus$ is an open
                immersion, Proposition~\ref{ref:openimmersions:prop}.
            \item if $f$ is a closed immersion, then $\fplus$ is a closed
                immersion, Proposition~\ref{ref:closedImmersions:prop}.
            \item\label{it:intro:last} if $f$ is formally \'etale (resp.~formally unramified), then
                $\fplus$ is also formally \'etale (resp.~formally unramified),
                Proposition~\ref{ref:etaleFormally:prop}.
        \end{enumerate}

Let us explain the idea of proof of
Theorem~\ref{ref:introRepresentability:thm}.  The case of affine $\varX$ is
discussed in Proposition~\ref{ref:affineRepresentability:prop} above. In the general
case, the main obstacle is topological: it is the lack of $\Group$-stable affine neighbourhoods.
For normal $\varX$ a Sumihiro-type result~\cite[Theorem~4.1]{Trautman} asserts
that every fixed point $x\in \varX^{\Group}$ has an affine $\Group$-stable
neighbourhood, however for general $\varX$ or for algebraic spaces this is no
longer true.

The solution is to look at the formal geometry of
$\varX$ along $\varX^{\Group}$ and of $\Gbar$ along $0\in \Gbar(\kk)$. Since
$\varX^{\Group} \subset \varX$ has a trivial $\Group$-action, also its
$n$-th thickening, denoted $\varX_n$, has a $\Group$-action (which is
non-trivial for $n > 0$).
Similarly, let $\Gbar_{n}$ be the $n$-th infinitesimal neighbourhood of
 $0_{\Gbar}\in \Gbar$. Then $\Gbar_n$ has an action of $\Group$ by left
 multiplication, so we may consider equivariant families $\Gbar_n \times S \to
 \varX$. The scheme-theoretic image of each such family lies in $X_n$.
The neighbourhoods $\{\Gbar_n\}_{n\in \NN}$ form an ascending family of closed
$\Group$-stable subschemes of $\Gbar$. We define a \emph{formal} version of the
\BBname{} functor $\Hfunctor{X}\colon \kSch^{\mathrm{op}}\to \Set$ as
\begin{equation}\label{eq:definitionFormal}
    \Hfunctor{X}(Z)=\left\{(\varphi_n)_{n\in \NN} \ |\ \varphi_n\colon \Gbar_n \times Z\to X,\,
        \mbox{ a $\Group$-equivariant morphism and }
        (\varphi_{n+1})_{|\Gbar_{n} \times Z}
        = \varphi_n\mbox{ for all }n\right\}.
\end{equation}
Informally, $\Hfunctor{X}(Z)$ parameterizes equivariant families defined on
the formal neighbourhood of $0_{\Gbar}\in \Gbar$.
Note that every $\Group$-equivariant morphism $\Gbar \times Z\to \varX$
restricts to $\Gbar_n\times Z\to \varX$ for all $n$. Hence, we obtain a
\emph{formalization} map $\Dfunctor{X}\to \Hfunctor{X}$.

\begin{theorem}[Algebraization of the \BBname{} functor]\label{ref:algebraizationABB}
    Let $\varX$ be a $\Group$-scheme locally of finite type over $\kk$. Then the restriction $\Dfunctor{X}\to \Hfunctor{X}$ is an
    isomorphism.
\end{theorem}
To prove Theorem~\ref{ref:introRepresentability:thm}, we first show that $\Hfunctor{X}$ is representable and that
$\Hfunctor{X} \to \varX^{\Group}$ is affine of finite type, see
Theorem~\ref{ref:representabilityformal:thm}. For this part the assumptions on
$\varX$ are minimal, we do not even assume locally finite type. Second,
Theorem~\ref{ref:algebraizationABB} asserts that $\Dfunctor{X}\to
\Hfunctor{X}$ is an isomorphism; here we do need that $\varX$ is locally of
finite type.

The proof of representability of $\Hfunctor{X}$ is quite intuitive: we prove
that each $\Hfunctor{X_n}$ is represented by a closed subscheme $Z_n \subset
X_n$. Then $Z_{n+1} \cap X_n = Z_n$ and so we obtain a formal scheme over
$Z_0$.
\[
    \begin{tikzcd}
        Z_0 \arrow[rd]\arrow[r, hook] & Z_1 \arrow[d]\arrow[r, hook] & Z_2
        \arrow[ld]\arrow[r, hook] &  \ldots\arrow[lld]\\
        & Z_0
    \end{tikzcd}
\]
Each projection $\iinfty{Z_n}\colon Z_n \to Z_0$ is affine, thus $Z_n=
\Spec_{Z_0}(\cA_n)$. A stabilization
results shows that for each fixed $\Group$-representation $V_{\lambda}$ the
map $\cA_n[\lambda] \to \cA_m[\lambda]$ is an
isomorphism for $n \geq m\gg 0$, where $\cA_n[\lambda]$ is the isotypic
component of $V_{\lambda}$ in $\cA_n$. Thus we
construct a sheaf of algebras $\cA$ on $Z_0$ by setting $\cA[\lambda] :=
\cA_n[\lambda]$ for each $\lambda$ and $n = n_{\lambda}\gg0$. The scheme $Z =
\Spec_{Z_0}(\cA)$ is the object representing $\Hfunctor{X}$.
We stress that $Z$ is effectively computed directly for an explicitly given $X$.

The proof of Theorem~\ref{ref:algebraizationABB} uses a recent deep and
elegant result of Alper-Hall-Rydh, which asserts that in \'etale topology and
over $\kkbar$ every fixed point has an affine $\Group$-stable neighbourhood.
This result requires $\varX$ to be locally of finite type and this is why
this assumption appears also in Theorem~\ref{ref:introRepresentability:thm}.
\begin{theorem}[{\cite[Theorem~2.6]{AHR}}]\label{ref:etalestableneighbourhoods:thm}
Let $X$ be a quasi-separated algebraic space, locally of finite type over an
algebraically closed field $\kk$, with an action of an affine algebraic group
$\Group$ over $\kk$. Suppose that $x$ is a $\kk$-rational point of $X$ with a
linearly reductive stabilizer $\Group_x$. Then there exists an affine scheme
$W$ together with action of $\Group$ and a $\Group$-equivariant {\'e}tale
neighbourhood $W\to X$ of $x$.
\end{theorem}
Note that this theorem is \emph{not} Luna's \'etale slice theorem: in
Theorem~\ref{ref:etalestableneighbourhoods:thm} the emphasis is on obtaining
an affine $\Group$-stable {\'e}tale neighbourhood, whereas Luna's \'etale slice theorem is usually
formulated for an affine $\Group$-scheme.
We note that for a normal variety $\varX$,
Theorem~\ref{ref:etalestableneighbourhoods:thm} follows from Sumihiro's
results~\cite{Sumihiro__equivariant}, see e.g.~\cite[Theorem, p.~54]{Trautman}.
We also note that for
quasi-projective variety $\varX$, Theorem~\ref{ref:etalestableneighbourhoods:thm}
follows from~\cite[Theorem~1.1(ii)]{Brion__linearization}.

There is an alternative proof of representability of $\Dfunctor{\cX}$. This
proof naturally extends to the case when $\cX$ is an algebraic stack locally
of finite type over $\kk$ and with an affine diagonal. We only briefly sketch
the idea of this approach below. For general information on stacks and on group actions on
stacks, we refer to~\cite{Olsson, Romagny__group_actions}.
This approach is rather abstract and gives no a priori information on the
geometry of $\cX^+$.
To begin with, in \cite[Theorem~1.6]{Halpern_Preygel} and
independently in~\cite[Theorem~2.22]{AHR} algebraicity of certain mapping
stacks is proven. In particular, it follows that the mapping stack
$\Maps([\Gbar/\Group], [\cX/\Group])$ is algebraic and locally of
finite type. The same is true for
$\cM := \Maps_{[\star/\Group]}([\Gbar/\Group], [\cX/\Group])$, where
$\star = \Spec \kk$ with trivial action. Now, $\cM$
parameterizes morphisms over $[\star/\Group]$. For each such morphism, by base
change $\star\to [\star/\Group]$ we obtain an equivariant morphism $\Gbar \to
\cX$. Hence $\cX^{+} := \Maps([\Gbar/\Group], [\cX/\Group])$ satisfies a universal
property analogous to~\eqref{eq:definitionFunctorial:new}. If $\cX$ is an
algebraic space, then the inertia groups of $\cX^+$ are trivial, hence it is
an algebraic space. If $\cX$ is a scheme of finite type, then $\cX^+$ is
\emph{quasi-compact}: by Sumihiro's result already $\widetilde{\cX}^+$ is quasi-compact, where
$\widetilde{\cX}^+$ is the \BBname{} decomposition of the normalization
$\widetilde{\cX}$ of $\cX$ with respect to the action of the normalization
$\widetilde{\Group}$ of $\Gbar$. Hence the morphism $\cX^{+} \to \cX$ is quasi-finite and
so $\cX^+$ is a scheme of finite type over $\kk$~\cite[Proposition~7.2.10]{Olsson}.
A much related and streamlined argument appears also in~\cite[Section~5.12]{AHR}, where Alper, Hall,
and Rydh recover Drinfeld's result as a consequence of existence of
equivariant Hom stacks \cite[Corollary~2.23]{AHR}.

The contribution of this work, as we see it, is the extension of \BBname{} decomposition
to actions of groups other than $\Gmult$ and setting its basic properties,
including representability.
All the above statements (apart from
Theorem~\ref{ref:etalestableneighbourhoods:thm}, of course) are, as far as we
know, novel and so are the properties~\eqref{it:intro:fst}-\eqref{it:intro:last} of
$f^+\colon \varX^+\to Y^+$. The proof of
Theorem~\ref{ref:introRepresentability:thm} differs from Drinfeld's proof even
in the $\Gmult$-case (and we have slightly weaker assumptions: \emph{locally
of finite type} instead of \emph{finite type}) and is more explicit than the
methods applying~\cite[Corollary~2.23]{AHR} or~\cite{Halpern_Preygel}.
The idea of formal \BBname{} functors is, in the $\Gmult$-case, due to
Drinfeld~\cite{Drinfeld}. The proof of
Theorem~\ref{ref:ABBdecomposition:thm:intro} is different from the original
proof of~\cite{BialynickiBirula__decomposition} and by affiness of $\Xplus\to
\varX^{\Group}$ reduces to the fact that smooth cones are necessarily affine
spaces. It would be interesting to see whether, perhaps by completely
different means, one could extend the above results to nonconnected groups or
reductive groups in characteristic $p$.

Now we briefly discuss the contents of the paper. In Section~\ref{sec:prelims}
and Section~\ref{sec:reductiveMonoids} we prove the necessary prerequisites on
algebraic groups and monoids. In particular, we show how to reduce to the case
where $\Gbar$ has a zero, see Theorem~\ref{ref:reductionToMonoidWithZero:thm}.
In Section~\ref{sec:actionsOnAffine} we deal with the (almost trivial case)
where $\varX$ is affine or covered by open affine $\Group$-stable
neighbourhoods. In Section~\ref{sec:BBproperties} we prove most of the
properties of $\fplus$ listed above. These proofs are used in
Section~\ref{sec:algebraization} to prove the main results.
Section~\ref{sec:smoothnessAndRelative} deals with the case of smooth $\varX$
or more generally a smooth morphism $X\to Y$.

\section{Preliminaries}\label{sec:prelims}

    Throughout when speaking about a group action, we mean a left action. We refer
    the reader e.g.~to \cite{Brion__Introduction} for a short introduction to algebraic
    groups actions over $\mathbb{C}$. In
	positive characteristic or over non-algebraically closed field, most questions become subtler,
	see~\cite{Jantzen__Representations}. However, the only groups in positive
	characteristic which we analyse here are tori and their finite extensions.

    We work over a field $\kk$. We do not assume that $\kk$ is algebraically
    closed or that it has characteristic zero. This lack of assumptions
    forces us to be careful with definitions, so we define most of standard
    concepts. By default the considered representations are rational, i.e.,
    unions of finite dimensional subrepresentations. We say that a group
    $\Group$ is \emph{(geometrically) linearly reductive} if every short exact sequence of
    $\Group_{\kkbar}$-representations splits. It is easily seen that if
    $\Group$ is geometrically linearly reductive, then every short exact
    sequence of $\Group$-representations also splits.  The
    group $\Group$ has a $\kk$-rational point $1_{\Group}$, hence it is connected if
    and only if it is geometrically connected. We will usually omit the
    word ``geometrically'' and speak about linearly reductive groups and
    connected groups.
    For $\kk$ algebraically closed of characteristic
    zero, these definitions agree with the usual ones.

    By a \emph{variety} we mean a reduced, separated, finite type scheme (which
    might be disconnected or reducible). A \emph{linear algebraic monoid} is
    an affine variety $\Gbar$ together with an associative multiplication $\mubar\colon
    \Gbar \times \Gbar \to \Gbar$ and a two-sided identity $1_{\Gbar}\in \Gbar(\kk)$. For an
    introduction to algebraic monoids, we refer to
    \cite{Putcha__Linear_algebraic_Monoids, Renner__Linear_algebraic_monoids}.
    We say that $\Gbar$ is \emph{normal} if every its connected component is
    an irreducible normal variety.
    The \emph{unit group} of a monoid is the open subset consisting of all
    invertible elements, denoted $\Gbar^{\times}$,
    see~\cite[Proposition~3.12]{Renner__Linear_algebraic_monoids}. It has a natural structure
    of an algebraic group and $\Gbar^{\times} \to \Gbar$ is a homomorphism of
    monoids.
    \begin{definition}
        Let $\Group$ be a linearly reductive group.
        A \emph{linearly reductive monoid} with unit group $\Group$ is an algebraic monoid
        $\Gbar$ whose unit group $\Gbar^{\times}  \simeq \Group$ is dense in $\Gbar$.
    \end{definition}

    \begin{example}
        Consider the monoid $\mathbb{A}^2$ with multiplication $(x_1,
        y_1)\cdot (x_2, y_2) = (x_1x_2, y_1y_2)$. Then the closed subscheme
        \[
            \Gbar = \{(x, x)\ |\ x\in \mathbb{A}^1\} \cup \left\{ (x, 0)\ |\ x\in
                \mathbb{A}^1 \right\}
        \]
        is a submonoid with $\Gbar^{\times} =
            \left\{ (x, x)\ |\ x\in \mathbb{A}^1 \setminus \{0\} \right\}
            \simeq
            \Gmult$. Hence $\Gbar^{\times} \subset \Gbar$ is reductive, but not dense.
    \end{example}
    Note that in the literature, usually a reductive monoid is
    connected~\cite[Definition~3.1]{Renner__Linear_algebraic_monoids}.
    Starting from Section~\ref{sec:BBproperties}, we will work with connected $\Gbar$,
    however for the time being we do \emph{not} impose any connectedness
    assumptions. Indeed, the aim of Section~\ref{sec:reductiveMonoids} is to
    show how to reduce, under mild assumptions, to the connected case.
    \begin{example}[{\cite[p.~158]{Vinberg__on_reductive_algebraic_semigroups}}]\label{ex:compactificationsOfGLtwo}
        Let $\Group = \kk^* \times \mathbf{SL}_2$. Then there are countably
        many non-isomorphic normal $\Gbar$'s. They are constructed as follows: the
        character lattice of the torus of $\Group$ is identified with
        $\mathbb{Q}e_1 + \mathbb{Q}e_2$, where $e_2$ comes from the torus of
        $\mathbf{SL}_2$, so that the Weyl chamber is $C := \{(e_1, e_2)\ |\ e_2\geq
        0\}$. For
        every $a\in \mathbb{Q}$ we consider the positive halfplane $K =
        \left\{ (e_1, e_2)\ |\ e_2\geq ae_1 \right\}$ and the intersection
        $K\cap C$. Then the subspace $A \subset H^0(\Group, \cO_{\Group})$
        spanned by all subrepresentations with highest weights in $K\cap C$ is a
        subring and $\Gbar = \Spec(A)$ is a normal linearly reductive monoid with unit
        group $\Group$.
    \end{example}
    An \emph{action} of a monoid $\Gbar$ is a morphism $\sigma\colon \Gbar
    \times \varX \to \varX$ satisfying the usual associativity and unitality
    conditions. A (rational)
    \emph{representation} of $\Gbar$ is a (rational) representation $V$ of
    $\Group$ such that the action map $\Group \times V\to V$ extends to
    $\Gbar \times V \to V$.

    For our considerations, it is essential to understand the
    structure of $\Group$-subrepresentations of $H^0(\Group,
    \cO_{\Group})$. This structure is well-known, but we recall it briefly,
    following~\cite{Brion__Introduction}.
    Assume that $\kk$ is algebraically closed and fix an irreducible
    $\Group$-representation $V$. Then $V^{\vee} = \Spec \Sym V$ and
    so
    \[
        \Hom(V, \GroupGlobalSects)^{\Group} \simeq
        \Maps^{\Group}(\Group, V^{\vee})  \simeq V^{\vee}.
    \]
    Moreover, $\GroupGlobalSects$ is a completely reducible
    $\Group$-representation, so
    \begin{equation}\label{eq:decomposition}
        \GroupGlobalSects  \simeq  \bigoplus_{V} V \otimes \Hom(V, \GroupGlobalSects)^{\Group} \simeq  \bigoplus_{V} V \otimes V^{\vee}
    \end{equation}
    where the summation goes over the set of isomorphism classes of
    irreducible $\Group$-representations.
    On the one hand, we have a canonical action of $\Gprod$ on $\Group$ where $(g_1, g_2)\in
    \Group(\kk) \times \Groupop(\kk)$ sends $g\in \Group(\kk)$ to $g_1 g g_2$.
    This induces a $(\Gprod)$-representation structure on $H^0(\Group,
    \cO_{\Group})$. On the other hand, we have a $(\Gprod)$-action on each $V
    \otimes V^{\vee}$, where $\Group$ acts on $V$ and $\Groupop$ acts on
    $V^*$. We leave to the reader checking that the
    isomorphism~\eqref{eq:decomposition} is $(\Gprod)$-equivariant.
    In particular, the isotypic component of an irreducible representation $V$
    is a \emph{simple} $(\Gprod)$-representation.

    We summarize the discussion with the following
    Proposition~\ref{ref:isotypiccomponents:prop}. In its statement we do
    \emph{not} assume that $\kk$ is algebraically closed.
    \begin{proposition}\label{ref:isotypiccomponents:prop}
        The $(\Gprod)$-representation $\GroupGlobalSects$ is a direct sum of
        pairwise non-isomorphic simple $(\Gprod)$-representations
        $V_{\lambda}$. Each of those $V_{\lambda}$ is equal to the isotypic component
        in $\GroupGlobalSects$ of a unique irreducible $\Group$-representation.
        Moreover, each $V_{\lambda}$ satisfies $\mu^{\#}(V_{\lambda}) \subset
        V_{\lambda} \otimes V_{\lambda}$.
    \end{proposition}
    \begin{proof}
        All but the last statement follows from the discussion above, after
        base change to $\kkbar$. Let us prove the last statement.
        For clarity, let $B = \GroupGlobalSects$.
        Fix a $(\Gprod)$-submodule $V \subset B$. Then
        $V$ is a left and right subrepresentation of $B$,
        hence $\mu^{\#}(V) \subset (V \otimes B) \cap (B \otimes V) =
        V\otimes V$.
    \end{proof}
    Here, the most important fact is that each $(\Gprod)$-representation
    appears with multiplicity \emph{one}. This will be used in particular in
    Section~\ref{sec:actionsOnAffine}.

    Finally, recall~\cite[\S(9.5)]{WedhornAG} that an open immersion $f\colon X\to Y$ is \emph{schematically
    dense} if $\cO_{Y} \to f_*\cO_X$ is an injective map of sheaves. If $X\to
    Y$ is schematically dense and factors as $X\to U\to Y$ where $U \subset Y$
    is open, then also $X\to U$ is schematically dense.

    \section{Linearly reductive monoids}\label{sec:reductiveMonoids}
    We work over an arbitrary field $\kk$.
    Fix a linearly reductive group $\Group$ over $\kk$.
    We do not assume that $\Group$ is connected.
    The following standard fact would be useful in our
    considerations.
    \begin{lemma}\label{com1}
        Let $X$ be an affine $\kk$-scheme with a $\Group$-action. Suppose that
        $F_1$ and $F_2$ are disjoint, $\Group$-stable closed subsets of
        $X$. Then there exists a $\Group$-invariant $f\in H^0(X,\cO_{X})$ such that $f(x)=1$ for $x\in F_1$ and $f(x)=0$ for $x\in F_2$.
    \end{lemma}
    \begin{proof}
    The map $H^0(X, \cO_{X})^{\Group} \to H^0(F_1 \cup
        F_2, \cO_{F_1\cup F_2})^{\Group} = H^0(F_1, \cO_{F_1})^{\Group}
        \oplus H^0(F_2, \cO_{F_2})^{\Group}$ is onto.
    \end{proof}

    \begin{theorem}\label{basechange}
        Let $\Group$ be a linearly reductive group over $\kk$ and $L$ be a
        field extension of $\kk$. Then $\Group_L$ is a linearly reductive group
        over $L$.
    \end{theorem}
    \begin{proof}
        In characteristic zero case this follows from the classification of
        reductive groups over $\kkbar$, while in positive characteristic it
        follows from Remark~\ref{rmk:positivejunkyard}
        below.
    \end{proof}

    \begin{remark}\label{rmk:positivejunkyard}
        Suppose $\charr \kk > 0$. This, by Nagata's result (\cite{Nagata}, see
        also~\cite{Nagata__userfriendly}), forces
        $\Group_{\kkbar}$ to be an extension of torus by a finite group of
        order non-divisible by $\charr \kk$. Let $\Groupconn \subset \Group$
        be the connected component of the identity. Then $\Groupconn_{\kkbar}$
        is a torus, hence it is of multiplicative type in the language
        of~\cite[Appendix~B]{Conrad}, see~\cite[Lemma~B.1.5]{Conrad}.
        Let $H \subset \Group$ be an algebraic subgroup. Then
        $H^{\circ}_{\kkbar} := H_{\kkbar}\cap
        \Groupconn_{\kkbar}$ is of multiplicative type
        by~\cite[Corollary~B.3.3]{Conrad}, so it is linearly
        reductive. Moreover, $H_{\kkbar}/H_{\kkbar}^{\circ} \into
        (\Group/\Groupconn)_{\kkbar}$ is finite of order non-divisible by $p$,
        hence is linearly reductive by Maschke's theorem.
        Thus, $H$ is an extension of linearly reductive groups,
        hence it is linearly reductive.
    \end{remark}

    \medskip
    We now argue that a linearly reductive monoid $\Gbar$ is the right
    analogue of $\mathbb{A}^1$. To do this, we introduce the notion of
    partial equivariant compactification. It captures
    the intuition of $\Gbar$ as an equivariant compactification of $\Group$.
    This notion is equivalent to the subsequent notion of a linearly reductive
    monoid, which will be used throughout the article.
    \begin{definition}
        Let $X$ be a \emph{finite type, affine} scheme over $\kk$ equipped with a schematically dense,
        open immersion $i\colon\Group\to X$. We say that $X$ is
        an \emph{equivariant partial compactification} of $\Group$ if it is
        equipped with left and right actions of $\Group$, denoted $\mu_l\colon \Group \times \varX\to
        \varX$ and $\mu_r\colon \varX \times \Group\to \varX$ such that $i$ is
        equivariant with respect to both of them.
    \end{definition}

    \begin{proposition}\label{ref:compactificationImpliesMonoid:prop}
        Let $X$ be an equivariant partial compactification of $\Group$. Then
        there exists a unique structure of a linearly reductive monoid
        on $X$ such that the inclusion $i\colon \Group \to X$ is a
        homomorphism of algebraic monoids.
        Moreover, $X^{\times} = i(\Group)$ for this monoid structure.
    \end{proposition}
    \begin{proof}
        Consider the following diagram induced by
        $i^{\#}$:
        \[
        \begin{tikzpicture}
            [description/.style={fill=white,inner sep=2pt}]
            \matrix (m) [matrix of math nodes, row sep=3em, column sep=2em,text height=1.5ex, text depth=0.25ex]
            { H^0(X,\cO_{X}) \otimes_{\kk} H^0(X,\cO_{X})&  &   H^0(X,\cO_{X}) \otimes_{\kk} H^0(\Group,\cO_{\Group})      \\
            H^0(\Group,\cO_{\Group}) \otimes_{\kk} H^0(X,\cO_{X})&   &
            H^0(\Group,\cO_{\Group})\otimes_{\kk}
            H^0(\Group,\cO_{\Group}).   \\} ;
            \path[->,font=\scriptsize]
            (m-1-1) edge node[above] {$1_{H^0(X,\cO_{X}) }\otimes_{\kk}i^{\#}$} (m-1-3)
            (m-2-1) edge node[above] {$1_{H^0(\Group,\cO_{\Group})}\otimes_{\kk}i^{\#}$} (m-2-3)
            (m-1-1) edge node[left] {$i^{\#}\otimes_{\kk}1_{H^0(X,\cO_{X})}$} (m-2-1)
            (m-1-3) edge node[right] {$i^{\#}\otimes_{\kk}1_{H^0(\Group,\cO_{\Group})}$} (m-2-3);
        \end{tikzpicture}
        \]
        It is a cartesian diagram of $\kk$-vector spaces. Therefore, the
        following diagram is cocartesian
        \begin{center}
            \begin{tikzpicture}
                [description/.style={fill=white,inner sep=2pt}]
                \matrix (m) [matrix of math nodes, row sep=3em, column sep=2em,text height=1.5ex, text depth=0.25ex]
                { \Group \times \Group&  &    \Group \times X           \\
                X\times \Group&   &       X\times X.   \\} ;
                \path[->,font=\scriptsize]
                (m-1-1) edge node[above] {$ $} (m-1-3)
                (m-2-1) edge node[above] {$ $} (m-2-3)
                (m-1-1) edge node[left] {$ $} (m-2-1)
                (m-1-3) edge node[right] {$ $} (m-2-3);
            \end{tikzpicture}
        \end{center}
        in the category of affine $\kk$-schemes. Hence, the morphisms $\mu_l\colon \Group \times \varX\to\varX$ and
        $\mu_r\colon \varX \times \Group \to \varX$ induce a morphism $\mu\colon
        X \times X \to X$, which restricts to the multiplication $\Group \times \Group \to
        \Group$. Since $\Group \times \Group$ is dense in $X \times X$ the
        map $\mu$ is associative and uniquely determined. Then
        the inclusion $j\colon \Group \to X^{\times}$
        is an open immersion. But it is also is a homomorphism of algebraic groups, hence
        its image is closed. By assumption the image of $j$ is dense. Summing
        up, $j$ is an open immersion with dense open-closed image. Thus $j$ is
        an isomorphism.
    \end{proof}

    We give some examples of partial equivariant compactifications.

    \begin{example}\label{ex:toric}
        Let $\Group = \Gmult^{\times n}$ be a split $n$-torus.
        Then the equivariant partial compactifications of $\Group$
        coincide with affine toric varieties with torus
        $\Group$.
    \end{example}

    \begin{example}\label{ex:matrices}
        Let $n$ be a positive integer and $\Group=\mathrm{GL}_n$ be the
        general linear group over $\kk$ of characteristic zero. Then the
        vector space $\mathbb{M}_n$ of
        $n\times n$ matrices over $\kk$ has a structure of an algebraic monoid
        $\kk$-scheme with respect to multiplication of matrices. The canonical
        open immersion $\mathrm{GL}_n\to \mathbb{M}_n$ is a morphism of monoid
        schemes over $\kk$ and the singleton of the zero matrix in
        $\mathbb{M}_n$ is the smallest nonempty, closed and
        $\mathrm{GL}_n$-stable subset of $\mathbb{M}_n$.
    \end{example}

    We now analyse the orbit structure on an algebraic monoid.
    \begin{lemma}\label{ref:smallestorbit:lem}
        Let $\Gbar$ be a linearly reductive monoid.
        There is a unique closed left $\Group$-orbit $\Fbar$ in $\Gbar$. In
        particular it is the intersection of all nonempty closed $\Group$-stable
        subschemes of $\Gbar$. It is also equal to the unique closed right $\Group$-orbit.
    \end{lemma}
    \begin{proof}
        Suppose there are two disjoint closed orbits in $\Gbar$, call them
        $O_1, O_2$. By Lemma~\ref{com1} there exists an $\Group$-invariant
        function $f$ on $X$ such that $f(O_1) = \{0\}$ and $f(O_2) = \{1\}$.
        But $H^0(\Gbar, \cO_{\Gbar})^{\Group} \subset H^0(\Group,
        \cO_{\Group})^{\Group} = \kk$, so the only invariant functions are
        constant. This is a contradiction. Thus there exists a unique closed
        left orbit $\Fbar_l$. Similarly, we see that there is a unique right $\Group$-orbit $\Fbar_r$, which is the intersection of all nonempty closed right $\Group$-equivariant subschemes.

        Assume that $\kk$ is algebraically closed. For every $g\in \Group(\kk)$ the set $\Fbar_l\cdot g$ is a left $\Group$-orbit of the same dimension as $\Fbar_l$ and hence $\Fbar_l=\Fbar_l\cdot g$ by uniqueness. Therefore, $\Fbar_l$ is right $\Group$-stable. This proves that $\Fbar_r\subseteq \Fbar_l$. By symmetry, $\Fbar_l\subseteq \Fbar_r$ and thus $\Fbar_l=\Fbar_r$. Now assume that $\kk$ is arbitrary. Pick the canonical morphism $p\colon \Gbar_{\kkbar}\to \Gbar$. Note that $p$ is a closed, surjective map which is both left and right equivariant. Hence it sends closed left orbits to closed left orbits (similarly for right orbits). Thus $\Fbar_l=\Fbar_r$.
\end{proof}
    In Example~\ref{ex:matrices} the unique closed orbit is just a point: the
    zero matrix. However, in general, even in the setup of
    Example~\ref{ex:toric}, it may happen that the orbit is not a point.
    A natural arithmetic question arises: does $\Fbar$ contain a
    $\kk$-point? Let us make a formal definition.
    \begin{definition}
        Let $\Gbar$ be a linearly linearly reductive monoid over $\kk$. We say that $\Gbar$
        \emph{is pointed} if and only if $\Fbar$ has a $\kk$-rational point. We say that
        $\Gbar$ is \emph{a monoid with zero} if there exists a $\kk$-point $0_G\in \Gbar$,
        such that $\Gbar\cdot 0_G = 0_G$. Then necessarily $\{0_G\} = \Fbar$, so
        $\Gbar$ is pointed.
    \end{definition}

    Now we make a small detour to discuss the important notions of Kempf torus
    and line, which we use in Theorem~\ref{pointedness} below to show that every
    $\Gbar$ over a perfect field is pointed.

    \newcommand{\oneparam}{\mathbb{G}_{m, \kk}}%
    \begin{theorem}\label{pointedness}
        Assume that $\kk$ is a perfect field and let $\Gbar$ be a linearly
        reductive monoid over $\kk$. Then there exists a one-parameter subgroup
        $\oneparam = \Spec \kk[t^{\pm 1}]$ and a homomorphism $\oneparam\to
        \Group$, which extends to a homomorphism of monoids $\varphi\colon
        \mathbb{A}^1_{\kk} = \Spec \kk[t]\to \Gbar$ such that $\varphi(1) =
        1_{\Group}$ and $\varphi(0)\in \Fbar$.
        In particular $\Gbar$ is pointed.
    \end{theorem}
    \begin{proof}
        \def\Gzero{\Group^0}%
        Since $\Group_{\kkbar}$ is linearly reductive, it is reductive, i.e., it
        has no non-zero unipotent normal subgroups. Hence
        also $\Group$ and its connected component of identity $\Gzero$ are reductive. Choose an embedding $\Gbar$ into a
        $\Group$-affine space $V$. Clearly the closure of the $\Group$-orbit of
        $1_{\Group}\in V$ contains $\Fbar$. Since $\Fbar$ is a $\Group$-orbit,
        also the closure of $\Gzero$ intersects $\Fbar$.
        By the Hilbert-Mumford criterion for perfect fields,
        see~\cite[Corollary~4.3]{Kempf_Hilbert_Mumford_perfect_field}, there exist
        a one-parameter subgroup $\oneparam$ \emph{over $\kk$} and
        $\varphi\colon \oneparam \to \Gzero$
        such that $\varphi(1) = 1_{\Group}$ and $\varphi(0)\in \Fbar$. Since
        $\varphi_{|\oneparam}$ is group homomorphism, $\varphi$ is a
        homomorphism of monoids.
    \end{proof}
    \begin{definition}
        Any $\oneparam \to \Group$ as in Theorem~\ref{pointedness} is called a
        \emph{Kempf's torus} and its compactification $\mathbb{A}^1_{\kk}\to
        \Gbar$ is called a \emph{Kempf's line}.
    \end{definition}
    For further use the following Corollary~\ref{ref:KempfTorus:cor} is
    important. It asserts that $H^0(\Gbar, \cO_{\Gbar})$ is
    \emph{non-negatively} graded for $\Gbar$ with zero. The non-negativity will be used in
    particular when proving the classical version of the \BBname{}
    decomposition.

    It is classical that a $\Gmult$-action on $\Spec(A)$ is equivalent to a
    $\mathbb{Z}$-grading on $A$. The same argument shows that an
    $\mathbb{A}^1$-action on $\Spec(A)$ is equivalent to an
    $\mathbb{N}$-grading.
    \begin{corollary}\label{ref:KempfTorusInvariants:cor}
        Let $\kk$ be perfect and let $\Gbar$ be a monoid with zero and with
        unit group $\Group$.
        Then any Kempf's torus
        $\oneparam \to \Group$ satisfies $\Gbar^{\oneparam} =
        0_{\Gbar}$. In particular, $H^0(\Gbar, \cO_{\Gbar})_0 = \kk$ in the
        induced $\mathbb{N}$-grading.
    \end{corollary}
    \begin{proof}
        Fix $\mathbb{G} = \oneparam$ and the
        homomorphism of monoids $\varphi\colon \mathbb{A}^1 \to \Gbar$ as in
        Theorem~\ref{pointedness}. Then $\varphi(0) = 0_{\Gbar}$. Moreover,
        \[
            \begin{tikzcd}
                \Phi\colon \mathbb{A}^1 \times \Gbar \arrow[r, "\varphi \times
                \id_{\Gbar}"] &\Gbar \times \Gbar \arrow[r, "\mu_{\Gbar}"]
                &\Gbar
            \end{tikzcd}
        \]
        is an $\mathbb{A}^1$-action on $\Gbar$. Pick an affine scheme $S$ and an
        $S$-point $g\in \Gbar^{\mathbb{G}}$. Then $\Phi(0, g) = \Phi(1, g)$ as
        $g$ is in the invariant locus. On the other hand $\Phi(0, g) =
        0_{\Gbar}\cdot g = 0_{\Gbar}$, as $0_{\Gbar}$ is an ideal of $\Gbar$.
        Thus $\Gbar^{\mathbb{G}} = 0_{\Gbar}$ as schemes.
    \end{proof}
    \begin{corollary}[grading induced by Kempf's torus]\label{ref:KempfTorus:cor}
        Let $\kk$ be perfect and $\Gbar$ be a linearly reductive monoid with
        zero. Let $W = \Spec B$ be an affine $\kk$-scheme with
        a $\Gbar$-action. Fix any Kempf's torus $\Gmult \to \Group$ and the
        induced grading $B = \bigoplus_{i\geq 0} B_i$.  Then $W^{\Group} =
        W^{\Gmult}$, so the ideal of $W^{\Group} \subset W$ is $B_+ :=
        \bigoplus_{i\geq 1} B_{i}$. In particular, $\bigcap_{n\in \NN} I(W^{\Group}
        \subset W)^n = 0$.
    \end{corollary}
    \begin{proof}
        Let $J \subset B$ be the ideal of $W^{\Group}$. Since $\Group$
        acts trivially on $W^{\Group}$, the weights of the Kempf's torus are
        zero on $B/J$. Thus $B_+ \subset J$. By restricting to
        $W^{\Gmult}$, we may assume $B_+ = 0$, i.e., $B =B_0$.
        The
        coaction $\Delta\colon B\to H^0(\Gbar, \cO_{\Gbar}) \otimes B$ is
        $\Gbar$-equivariant, hence also $\mathbb{A}^1$-equivariant. Thus
        $\Delta(B) \subset H^0(\Gbar, \cO_{\Gbar})_0 \otimes B$. By
        Corollary~\ref{ref:KempfTorusInvariants:cor}, we have $H^0(\Gbar, \cO_{\Gbar})_0 \otimes B =
        \kk\otimes B$. This shows that the action $\Gbar \times W\to W$ factors as
        \[
            \begin{tikzcd}
                \Gbar \times W \arrow[r, "pr_1 \times \id_W"] & 0_{\Gbar}
                \times W \arrow[r, ] & W,
            \end{tikzcd}
        \]
        thus $W = W^{\Group}$.
    \end{proof}

    We now give a characterization of pointed reductive monoids in terms of
    subgroups.
    \begin{proposition}\label{ref:Nroupdefinition:prop}
        Let $\Gbar$ be a linearly reductive monoid over $\kk$. Then $\Gbar$ is pointed if
        and only if there exist a normal subgroup scheme $\Nroup\subset \Group$,
        possibly non-reduced, and
        an isomorphism $\Fbar  \simeq \Group/\Nroup$ of left
        $\Group$-varieties. If this holds, then $\Nroup$ is the stabilizer of
        any $\kk$-rational point of $\Fbar$.
    \end{proposition}
    \begin{proof}
        If $\varphi\colon \Group/\Nroup \to \Fbar$ is an isomorphism then
        $\varphi(1_{G})$ is the $\kk$-point of $\Fbar$. Conversely, having
        $x\in \Fbar(\kk)$ we can restrict the multiplication on $\Gbar$ to
        $f\colon\Group \times \{x\} \to \Fbar$. The map $f$ is equivariant and surjective, hence
        it induces an isomorphism $\Fbar \simeq \Group/\Nroup$, where $\Nroup$
        is the stabilizer of $x$. It remains to prove that $\Nroup$ is normal.
        If $\charr \kk > 0$, then $\Nroup$ is normal, since $\Group$ is
        commutative by assumption. If $\charr \kk = 0$, then $\Group$ is smooth,
        hence it is enough to prove that $\Nroup(\kkbar) \subset
        \Group(\kkbar)$ is normal. Thus we assume that $\kk$ is algebraically
        closed. By Lemma~\ref{ref:smallestorbit:lem} we
        have that $\Fbar$ is the (unique) left $\Group$-orbit
        of minimal dimension and that it is also a right
        $\Group$-orbit. Hence, for every $y\in \Fbar(\kk)$ and
        $n\in \Nroup(\kk)$ we have an element $g\in \Group(\kk)$ such
        that $y = xg$ and thus $ny = n(xg) = (nx)g = xg = y$. Therefore,
        $\Nroup$ is the kernel of $\Group\to
        \mathrm{Aut}(\Fbar)$ so it is normal.
    \end{proof}

    In the remaining part of this section we will prove that the \BBname{}
    functor for a normal, pointed, reductive monoid is isomorphic to a
    \BBname{} functor for a certain submonoid with zero, whose unit group is
    connected. This motivates the restriction to such monoids made in the
    introduction.
    \begin{assumption}\label{assumptions:forInduction}
        For the remaining part of the present section, we assume that $\Gbar$ is
        a normal (as a variety), pointed, reductive monoid.
    \end{assumption}
    We stress that Assumption~\ref{assumptions:forInduction} does not include connectedness, we merely require that
    each connected component of $\Gbar$ is irreducible and normal.

    Let $\Nroup$ be the stabilizer of any $\kk$-rational point of $\Fbar$, as
    in Proposition~\ref{ref:Nroupdefinition:prop}. Let $\Nred \subset \Nroup$
    be the connected component of identity in the reduction of $\Nroup$.
    \begin{lemma}\label{ref:Nlinred:lem}
        The group $\Nred$ is geometrically linearly reductive. Also the
        functor $(-)^{\Nroup}$ is exact.
    \end{lemma}
    \begin{proof}
        In positive characteristic both assertions follow from classification,
        Remark~\ref{rmk:positivejunkyard}. In characteristic zero $\Nroup$ is
        reduced and both $\Nroup$ and $\Nred$ are reductive normal subgroups of a reductive group
        $\Group$, or by Matsushima's theorem, see
        e.g.~\cite{Matsushima__affine_quotients,
        Cline_Parshall_Scott__affine_quotients}.
    \end{proof}

    The immersions $\Fbar \subset
    \Gbar$ and $\Group \subset \Gbar$ induce the following morphisms of quotients:
    \newcommand{\pibar}{\overline{\pi}}%
    \begin{equation}\label{eq:diagramquotients}
            \begin{tikzcd}
                 \Fbar \arrow[r, "c"]\arrow[d] & \Gbar\arrow[d, "\pibar"] &
                 \Group\arrow[l, "i"']\arrow[d]\\
                 \Fbar\goodquotient \Nroup \arrow[r]& \Gbar\goodquotient\Nroup
                 & \arrow[l]\Group/\Nroup.
            \end{tikzcd}
    \end{equation}
    Since $\Nroup$ acts trivially on $\Fbar$, the map $\Fbar\to \Fbar \goodquotient
    \Nroup$ is an isomorphism.
    \begin{proposition}\label{ref:isoOnQuotients:prop} The morphisms
        $\Fbar\goodquotient\Nroup\to \Gbar\goodquotient\Nroup$ and
        $\Group/\Nroup\to \Gbar\goodquotient\Nroup$ are isomorphisms.
    \end{proposition}
    \begin{proof}
        \def\Nredloc{\Nroup}%
        By
        Lemma~\ref{ref:Nlinred:lem}, the functor $(-)^{\Nredloc}$ is exact.
        The morphisms $c$, $i$ induce, respectively, a surjection $c^{\#}\colon H^0(\Gbar,
        \cO_{\Gbar})^{\Nredloc} \to H^0(\Fbar, \cO_{\Fbar})^{\Nredloc}$ and an
        injection $i^{\#}\colon H^0(\Gbar, \cO_{\Gbar})^{\Nredloc}\to H^0(\Group,
        \cO_{\Group})^{\Nredloc}$, both
        left $\Group$-equivariant.
        Let $V$ be an irreducible $\Group$-representation and $V_{\Fbar}$, $V_{\Gbar}$,
        $V_{\Group}$ be the $V$-isotypic components of $H^0(\Fbar,
        \cO_{\Fbar})^{\Nredloc}$, $H^0(\Gbar, \cO_{\Gbar})^{\Nredloc}$ and $H^0(\Group,
        \cO_{\Group})^{\Nredloc}$ respectively. Since $c^{\#}$ is surjective and
        $i^{\#}$ is injective, we have inequalities
        $\dim V_{\Fbar} \leq \dim
        V_{\Gbar} \leq \dim V_{\Group}$.
        Proposition~\ref{ref:isotypiccomponents:prop} implies that $\dim
        V_{\Group} \leq \infty$.
        But $\Nredloc$ acts trivially on $\Fbar \simeq
        \Group/\Nredloc$, so
        \[
            H^0(\Fbar, \cO_{\Fbar})^{\Nredloc} = H^0(\Fbar, \cO_{\Fbar}) =
            H^0(\Group/\Nredloc, \cO_{\Group/\Nredloc})\simeq H^0(\Group,
            \cO_{\Group})^{\Nredloc}
        \]
        as $\Group$-representations.
        Hence $\dim V_{\Group} = \dim V_{\Fbar}$, thus both $c^{\#}$ and
        $i^{\#}$
        are isomorphisms when restricted to the isotypic component of $V$.
        Since $V$ was arbitrary, both $c^{\#}$ and $i^{\#}$ are isomorphisms.
    \end{proof}
    The following key Proposition~\ref{ref:Nproperties:prop} is proven
    in~\cite[Proposition~12]{Brion__On_algebraic_semigroups_and_monoids}. We
    reproduce it here both for completeness and to show that the proof holds
    in our generality (in principle,~\cite{Brion__On_algebraic_semigroups_and_monoids} works
    over algebraically closed field, however the proof there does not use this
    assumption).
    \begin{proposition}\label{ref:Nproperties:prop}
        The group scheme $\Nroup$ is smooth and connected, hence $\Nroup =
        \Nred$.
    \end{proposition}
    \begin{proof}
        We want to prove that $\Nred = \Nroup$. By
        Proposition~\ref{ref:isoOnQuotients:prop} the inclusion $\Group/\Nroup
        \to\Gbar\goodquotient \Nroup$ is an isomorphism.
        Consider the natural map $\Group/\Nred \to \Group
        /\Nroup$ and the associated cartesian diagram
        \begin{equation}\label{eq:diagramNproperties}
            \begin{tikzcd}
                Z \arrow[r]\arrow[d] &
                \Gbar \arrow[d, "\pibar"]\\
                \Group/\Nred\arrow[r] &
                \Group/\Nroup
            \end{tikzcd}
        \end{equation}
        Let $\Group\into Z$ be the embedding obtained from the natural maps
        $\Group \into \Gbar$ and $\Group \to \Gbar/\Nred$. Let $Z'$ be the
        closure of $\Group$ in $Z$.
        The surjection $\Group/\Nroup \to \Group/\Nred$ is a finite map, hence
        also $Z\to \Gbar$ is finite, so $Z' \to \Gbar$ is finite. The latter
        map is also
        birational, hence is an isomorphism, as $\Gbar$ is normal. We obtain a
        map
        \[
            \varphi\colon \Gbar \to Z' \to \Group/\Nred.
        \]
        By construction, all maps in~\eqref{eq:diagramNproperties} are
        homomorphisms of monoids. In particular, $\varphi$ is such. Fix any
        $\kk$-scheme $S$ and $n\in\Nroup(S)$. Let $x\in \Fbar(\kk)$, which
        exists by pointedness of $\Gbar$. Then $\Nroup$ stabilizes $x$. View
        $x$ as a constant $S$-point of $\Fbar$.
        Then $\varphi(n)\varphi(x) = \varphi(nx) = \varphi(x)$.
        But $\varphi(x)$ is invertible in the group $(\Group/\Nred)(S)$, so
        $\varphi(n) = 1_{\Group/\Nred(S)}$.
        Since $\varphi$ restricts to the canonical surjection $\Group \to
        \Group/\Nred$, we have $\varphi(n) = n\in
        (\Group/\Nred)(S)$.
        Thus, $n\in \Nred(S)$. As $n$ was an arbitrary $S$-point, we
        have $\Nroup \subset \Nred$. This proves that $\Nroup$ is reduced and
        connected, hence smooth connected.
    \end{proof}

    \begin{remark}
        In the setting of Proposition~\ref{ref:Nproperties:prop}, both
        connectedness and
        reducedness of $\Nroup$ are nontrivial, as exhibited
        in~\cite[Example~3, p.~23]{Brion__On_algebraic_semigroups_and_monoids}. We recall this example briefly. Let
        \[
            \Gbar = \left\{ (x, y, z)\in \mathbb{A}^3\ |\ xy^n = z^n, x\neq 0
        \right\}.
        \]
        This is a monoid via coordinatewise multiplication. Its
        unit group is
        \[
            \Group = \left\{ (x, y, z)\in \Gbar \ |\ z\neq
            0\right\},
        \]
        which is isomorphic to $\Gmult^{2}$ via projection from
        the first coordinate. We have $\Fbar = V(y, z)\subset \Gbar$ with
        $\kk$-point $(1, 0, 0)$ and so $\Nroup = \left\{
        (y, z)\ |\ y^n = z^n\right\}  \simeq \Gmult \times \mu_n$, which
        is disconnected (unless $n$ is a power of $\charr(\kk)$) and, for
        positive $\charr(\kk)$ dividing $n$, non-reduced.
    \end{remark}
    \begin{corollary}\label{ref:fibration:cor}
        The map $\pibar\colon \Gbar\to \Group/\Nroup$ is a flat surjective morphism of
        algebraic monoids.
    \end{corollary}
    \begin{proof}
        Clearly $\pibar_{|\Group}$ is a homomorphism of algebraic monoids (in fact,
        groups). Since $\Group \subset \Gbar$ is dense, also $\pibar$ is a
        homomorphism of algebraic monoids. It is surjective, $\Group$-equivariant, and
        the target $\Group/\Nroup = \Group/\Nred  \simeq \Fbar$ is reduced
        with transitive $\Group$-action. Thus, $\pibar$ is flat by generic
        flatness.
    \end{proof}

    We now construct a submonoid $\Nbar
    \subset \Gbar$ which has a zero and is such that the \BBname{} functors of
    $\Gbar$ and $\Nbar$ are isomorphic. This construction is called
    \emph{induction} and is performed in greater generality
    in~\cite[\S3.2]{Brion__On_algebraic_semigroups_and_monoids}.
    From Proposition~\ref{ref:isoOnQuotients:prop}, we have a morphism
    \[\pibar\colon \Gbar \to \Gbar\goodquotient \Nroup  \simeq \Group/\Nroup.\]
    We will prove that
    $\Gbar$ as a $\Group$-variety is a semidirect product of the fiber of
    $\pibar$ and its target.
    We define
    \begin{equation}\label{eq:Nbardef}
        \Nbar := \pibar^{-1}(1_{\Group/\Nroup}).
    \end{equation}
    \newcommand{\Ndiag}{\Nroup^{\mathrm{diag}}}%
    Note that $\Nbar \cap \Group = \Nroup$ and that $\Nbar$ is a submonoid of
    $\Gbar$. Eventually, we will see that  $\Nbar$
    is simply the
    closure of $\Nroup$ in $\Gbar$.
    We have a natural left
    $\Group$-equivariant morphism
    \begin{equation}\label{eq:actiononNbar}
        j\colon \Group \times \Nbar \to \Gbar,
    \end{equation}
    given on points by $(g, \bar{n})\mapsto g\bar{n}$. ($j$ is usually \emph{not} a homomorphism
    of monoids as $\Nbar$ is not necessarily central in $\Gbar$).
    We also have an embedding $e\colon \Nroup \into \Group \times \Nbar$, given on points
    by $n\to (n^{-1}, n)$. Let $\Ndiag = e(\Nroup)$. Then $\Ndiag$ acts
    on $\Group \times \Nbar$ by the rule $n\cdot (g, \bar{n}) = (gn^{-1},
    n\bar{n})$ and the morphism~\eqref{eq:actiononNbar} factors
    through the quotient and gives a left $\Group$-equivariant morphism of
    schemes over $\Group/\Nroup$:
    \begin{equation}\label{eq:Nbarisomorphism}
        \begin{tikzcd}
            (\Group \times \Nbar)/\Ndiag \arrow[rr, "\bar{j}"]\arrow[rd, "pr_1"] &&\Gbar\arrow[ld, "\pibar"]\\
            & \Group/\Nroup  &
        \end{tikzcd}
    \end{equation}
    \begin{proposition}\label{ref:isoOnMonoids:prop}
        The morphism $\bar{j}$ from Diagram~\eqref{eq:Nbarisomorphism} is an
        isomorphism. In particular
        \[
            \Gbar  \simeq (\Group \times \Nbar)/\Ndiag
        \]
        as left $\Group$-varieties. Moreover, $\Nbar$ is reduced and
        irreducible, normal, linearly reductive monoid with zero.
    \end{proposition}
    \begin{proof}
        Let us step back and consider first the diagram
        \begin{equation}\label{eq:fppftrivialization}
            \begin{tikzcd}
                \Group \times \Nbar\arrow[r, "j"]\arrow[d, "pr_1"] & \Gbar\arrow[d, "\pibar"]\\
                \Group \arrow[r, "/\Nroup"] & \Group/\Nroup
            \end{tikzcd}
        \end{equation}
        We claim that~\eqref{eq:fppftrivialization} is cartesian. This we
        check directly. Namely, let $\alpha\colon Z\to \Gbar$ and $\beta\colon
        Z\to \Group$ be two morphisms, which agree on $\Group/\Nroup$. Let
        $\gamma\colon Z\to \Gbar$ be defined as $\gamma(z) = \left(\beta(z)\right)^{-1}\cdot\alpha(z)$. This is a morphism such
        that $\pibar
        \circ \gamma$ is the constant map $1_{\Group/\Nroup}$. Hence we obtain
        $\gamma\colon Z\to \Nbar$ and thus $(\beta, \gamma)\colon Z\to \Group
        \times \Nbar$, which proves that the diagram is cartesian.
        We note that the isomorphism $\Group
        \times_{\Group/\Nroup} \Gbar\to \Group \times \Nbar$ on the level of
        points is
        \begin{equation}\label{eq:Naction}
            (g, \bar{h}) \mapsto (g, g^{-1}\bar{h}).
        \end{equation}
        The pullback $\Group\times_{\Group/\Nroup} \Gbar$ has a pulled-back
        $\Nroup$-action: the group $\Nroup$ acts on the first coordinate.
        By~\eqref{eq:Naction} this translates into the $\Ndiag$-action on
        $\Group \times \Nbar$. Moreover, $j$ is a pullback of $(-/\Nroup)$, hence
        is smooth by
        Proposition~\ref{ref:Nproperties:prop}. Thus, $\Group \times \Nbar$ is
        a normal scheme (i.e., every its connected component is an irreducible normal
        variety), hence also $\Nbar$ is reduced and a normal
        scheme \cite[Tags~034E, 034F]{stacks_project}.
        Moreover, $(-)^{\Nroup}$ is exact by Lemma~\ref{ref:Nlinred:lem}, so
        the square~\eqref{eq:functorialQuotient} below is cartesian
        \begin{equation}\label{eq:functorialQuotient}
            \begin{tikzcd}
                \Nbar \arrow[r, hook]\arrow[d, "\pibar_{|\Nbar}"] & \Gbar\arrow[d, "\pibar"]\\
                \Nbar\goodquotient\Nroup \arrow[r, hook] & \Gbar\goodquotient
                \Nroup.
            \end{tikzcd}
        \end{equation}
        Hence $\Nbar\goodquotient \Nroup = \pibar(\Nbar) = \{1_{\Group}\}$
        is a point. The group $\Nroup$ is connected by
        Proposition~\ref{ref:Nproperties:prop}, hence acts on each component
        of $\Nbar$ separately, so the quotient has at least as many components
        as $\Nbar$. But the quotient is a point, hence $\Nbar$ is connected.

        Now, let us look at the cartesian square~\eqref{eq:fppftrivialization}
        from different point of view. The map $\pibar$ is faithfully flat by
        Corollary~\ref{ref:fibration:cor}. Therefore, $\Group \times \Nbar  \to \Gbar$ is a
        pullback of the quotient $\Group \to \Group/\Nroup$ by a faithfully
        flat map $\pibar$, hence is a quotient by the induced action of $\Nroup$
        itself. Thus $(\Group \times \Nbar)/\Ndiag \to \Gbar$ is an
        isomorphism.
    \end{proof}

    The following theorem allows us to reduce the considerations of \BBname{}
    functors from normal, pointed linearly reductive monoid to the case of
    linearly reductive monoids \emph{with zero and with connected unit
    group}.  It summarizes the constructions made above. We restate the
    assumptions listed in~\ref{assumptions:forInduction} in the theorem.
    \begin{theorem}\label{ref:reductionToMonoidWithZero:thm}
        Let $\Group$ be a linearly reductive group and $\Gbar$
        be an algebraic monoid with unit group $\Group$, which is normal as a
        variety. Let $\Fbar$ be the
        intersection of all closed left $\Group$-stable subschemes. Assume that $\Fbar(\kk)$ is
        non-empty (this is automatic for $\kk$ perfect). Let $\Nroup$ be the
        stabilizer of a point in $\Fbar(\kk)$ and let $\Nbar$ be its closure
        in $\Gbar$.
        Then
        \begin{enumerate}
            \item\label{eq:assertionsN} the group $\Nroup$ is linearly reductive and
                geometrically connected.
            \item\label{eq:assertionsNbar} $\Nbar$ is an linearly reductive monoid with zero and with unit group
                $\Nroup$. Moreover, $\Nbar$ is normal and irreducible.
            \item\label{eq:assertionBBequivalence} for every scheme $\varX$
                with $\Group$-action, the restriction of families
                \begin{equation}\label{eq:BBfunctorAgreeOne}
                    \left( \varphi\colon \Gbar \times S \to \varX\right) \mapsto
                    \left(\varphi_{|\Nbar}\colon \Nbar \times S\to \varX\right).
                \end{equation}
                gives an \emph{isomorphism} of functors
                $\Dfunctor{\varX, \Gbar} \to \Dfunctor{\varX, \Nbar}$. This
                isomorphism is functorial in $\varX$.
        \end{enumerate}
    \end{theorem}
    \begin{proof}
        Assertion~\eqref{eq:assertionsN} is proved in
        Lemma~\ref{ref:Nlinred:lem} and
        Proposition~\ref{ref:Nproperties:prop}. Assertion~\eqref{eq:assertionsNbar}
        is proved in Proposition~\ref{ref:isoOnMonoids:prop}. It remains to prove
        Assertion~\eqref{eq:assertionBBequivalence}. In down-to-earth terms, we
        provide an equivalence between $\Gbar$-equivariant and
        $\Nbar$-equivariant families on $\varX$.

        The passage from $\Gbar$ to $\Nbar$ is just the
        restriction~\eqref{eq:BBfunctorAgreeOne}.
        In the passage from $\Nbar$ to $\Gbar$, the key point is the
        isomorphism of $\Group$-schemes $\Gbar  \simeq (\Group \times
        \Nbar)/\Nroup$, which is proved in Proposition~\ref{ref:isoOnMonoids:prop}.
        Let $\sigma\colon \Gbar \times \varX\to \varX$ be the action and $p\colon \Group
        \times \Nbar \to \Gbar$ be the quotient map.
        For a family $\psi\colon \Nbar \times S \to \varX$ we fix
        $\widetilde{\psi}\colon \Group \times \Nbar \times S \to \varX$, which
        is the composition
        \[
            \begin{tikzcd}
                \Group \times \Nbar \times S \arrow[r, "(\id_{\Group}{,} \psi)"] & \Group \times \varX \arrow[r, "\sigma"] & \varX.
            \end{tikzcd}
        \]
        But $\psi$ is $\Nroup$-equivariant, which translates into the fact that
        $\widetilde{\psi}$ factors through $p\colon \Group
        \times \Nbar \to \Gbar$. Thus $\widetilde{\psi}$ induces $\hat{\psi}\colon \Gbar \times S \to
        \varX$. Then $\hat{\psi}$ is the requested $\Gbar$-equivariant family.
        The described operations are clearly functorial in $S$ and inverse to
        each other, thus they give an isomorphism $\Dfunctor{\varX, \Gbar}
        \to \Dfunctor{\varX, \Nbar}$.
    \end{proof}

    For further use we note the following property of linearly reductive monoids with zero.
    \begin{lemma}\label{ref:zerobasechange:lem}
        Let $\Gbar$ be a linearly reductive monoid over $\kk$ with zero and $L \supset
        \kk$ be a field. Then $\Gbar_{L}$ is also a reductive monoid with zero
        and in fact $0_{\Gbar_{L}}$ is the preimage of $0_{\Gbar}$ under the
        natural projection.
    \end{lemma}
    \begin{proof}
        By Theorem~\ref{basechange} the group $\Group_L$ is linearly
        reductive, so $\Gbar_{L}$ is reductive. We apply Lemma~\ref{ref:smallestorbit:lem} to
        $\Gbar_{L}$ and deduce that the intersection of all non-empty
        $\Group_L$-stable closed subsets is non-empty.
        Let $\pi\colon \Gbar_{L} \to \Gbar$ be the projection. Since
        $0_{\Group}$ is $\kk$-rational, its preimage under $\Gbar_L\to \Gbar$
        is equal to a single $L$-rational point $o_L$, which is closed and
        invariant. Hence $o_L$ is the zero of $\Gbar_L$.
    \end{proof}

    \section{Group and monoid actions on affine
    schemes}\label{sec:actionsOnAffine}

        Let $\Gbar$ be a linearly reductive monoid. In this section we
        analyse the actions of $\Group$ and $\Gbar$ on affine schemes and
        schemes with $\Group$-stable affine charts. All results are very
        general; we do not need or require that $\Group$ be connected, $\Gbar$
        have zero etc.

        \begin{corollary}\label{ref:decompositionOfGlobalSectionsGbar:cor}
            As a $(\Gprod)$-module, the space $H^0(\Gbar, \cO_{\Gbar})$ decomposes
            into pairwise non-isomorphic simple $(\Gprod)$-modules. In
            particular $H^0(\Gbar, \cO_{\Gbar}) \subset H^0(\Group,
            \cO_{\Group})$ is a $(\Gprod)$-subrepresentation.
        \end{corollary}
        \begin{proof}
            The morphism $\Group \into\Gbar$ is $(\Gprod)$-equivariant.
        \end{proof}

        \begin{definition}
            Let $V$ be an irreducible $\Group$-representation. We say that $V$
            is a \emph{$\Gbar$-representation} if and only if $V$ appears as a
            subrepresentation of $H^0(\Gbar, \cO_{\Gbar})$. Otherwise, we say
            that $V$ is an \emph{outsider representation}.
        \end{definition}
        The definition is justified by the following lemmata.
        \begin{lemma}\label{ref:GbarSubreprsLieInGbar:lem}
            Let $V \subset H^0(\Group, \cO_{\Group})$ be an irreducible
            $\Group$-subrepresentation which is a $\Gbar$-representation.
            Then $V \subset H^0(\Gbar, \cO_{\Gbar}) \subset H^0(\Group,
            \cO_{\Group})$.
        \end{lemma}
        \begin{proof}
            Let $0\neq v\in V$ and let $W \subset H^0(\Group, \cO_{\Group})$
            be an irreducible $(\Gprod)$-representation containing $v$. Then
            $V \subset W$.  By
            Proposition~\ref{ref:isotypiccomponents:prop}, $W$ is the
            isotypic component of the representation $V$.
            Since $V$ is a $\Gbar$-representation, there exists an irreducible
            subrepresentation $V' \subset H^0(\Gbar, \cO_{\Gbar})$ isomorphic
            to $V$. Thus $W\cap H^0(\Gbar, \cO_{\Gbar}) \neq 0$. But $W$ is
            simple, and $H^0(\Gbar, \cO_{\Gbar})$ is a
            $(\Gprod)$-subrepresentation by
            Corollary~\ref{ref:decompositionOfGlobalSectionsGbar:cor}. Thus
            $W$ is contained in $H^0(\Gbar, \cO_{\Gbar})$ and in particular $V$ is
            contained there.
        \end{proof}
        \begin{lemma}\label{ref:mainRepr:lem}
            Let $V$ be a representation of $\Group$. Then the following are
            equivalent:
            \begin{enumerate}
                \item\label{it:ReprsGroupOne} the representation $\Group \times V\to V$ extends to
                    $\Gbar \times V\to V$.
                \item\label{it:ReprsGroupTwo} The coaction $V\to
                    H^0(\Group,\cO_{\Group})\otimes V$ factorizes as
                    $V\to H^0(\Gbar,\cO_{\Gbar})\otimes V\into
                    H^0(\Group,\cO_{\Group})\otimes V$.
                \item\label{it:ReprsGroupThree} The representation $V$ contains no outsider
                    representations.
            \end{enumerate}
        \end{lemma}
        \begin{proof}
            \def\Wbar{\overline{W}}%
            The equivalence of~\eqref{it:ReprsGroupOne}
            and~\eqref{it:ReprsGroupTwo} is straightforward.
            Let $W = H^0(\Group, \cO_{\Group}) \otimes V$ and $\Wbar :=
            H^0(\Gbar,\cO_{\Gbar})\otimes V$. Then $\Wbar \subset W$ are
            $\Group$-representations, where $\Group$ acts on
            $H^0(\Gbar,\cO_{\Gbar}) \subset H^0(\Group, \cO_{\Group})$ naturally and on $V$ \emph{trivially}.
            Assume that \eqref{it:ReprsGroupOne} and \eqref{it:ReprsGroupTwo}
            hold.
            The (extended) action map $\Gbar \times V\to V$ is
            $\Group$-equivariant, where $\Group$ acts on $\Gbar \times V$ by
            left multiplication on $\Group$ and trivial action on $V$. Hence
            the (extended) coaction $\Delta\colon V\to
            H^0(\Gbar,\cO_{\Gbar})\otimes V = \Wbar$ is $\Group$-equivariant.
            Moreover, $\Delta$ is injective as $1_{\Gbar} \times V\to V$ is an
            isomorphism. Thus $V$ is isomorphic to a subrepresentation of
            $\Wbar$. Since $\Wbar$ has no outsider representations, also $V$ has none,
            which proves~\eqref{it:ReprsGroupThree}.
            Assume~\eqref{it:ReprsGroupThree}. Consider the coaction map
            $\Delta\colon V\to H^0(\Group,\cO_{\Group})\otimes V = W  \simeq
            H^0(\Group,\cO_{\Group})^{\dim V}$. By
            Lemma~\eqref{ref:GbarSubreprsLieInGbar:lem} we have $\Delta(V)
            \subset H^0(\Gbar,\cO_{\Gbar})\otimes V$, which
            proves~\eqref{it:ReprsGroupOne} and concludes the proof.
        \end{proof}

        \begin{proposition}\label{ref:representabilityForAffine:prop}
            Let $\varX = \Spec(A)$ be an affine $\kk$-scheme with a
            $\Group$-action $\sigma\colon \Group \times \varX\to \varX$. Then
            the functor $\DX$ is represented by a closed subscheme
            $\ioneX\colon \Xplus\into \varX$
            defined by the ideal generated by all outsider representations in
            $A$.
        \end{proposition}
        \begin{proof}
            Let $Z$ be the closed subscheme defined by the ideal generated by all outsider
            representations.
            Let $S = \Spec(C)$ be any affine scheme and $f\colon \Gbar \times
            S\to \varX$ be an equivariant morphism.
            The corresponding $f^{\#}\colon A \to H^0(\Gbar, \cO_{\Gbar})
            \otimes C$ is a morphism of $\Group$-representations, where
            $\Group$ acts usually on $H^0(\Gbar, \cO_{\Gbar})$ and trivially
            on $C$. Clearly $H^0(\Gbar, \cO_{\Gbar})\otimes C$ contains only $\Gbar$-representations, hence
            all outsider representations lie in $\ker f^{\#}$. Thus $f$
            factors through $Z$ and we obtain $f\colon\Gbar \times S\to Z$. The algebra $A' = H^0(Z,
            \cO_{Z})$ by construction contains no outsider representations,
            hence by
            Lemma~\ref{ref:mainRepr:lem} the action on $Z$ extends to
            $\sigmabar_Z\colon \Gbar
            \times Z\to Z$. Consider Diagram~\eqref{eq:extendingAction} below.
            \begin{equation}\label{eq:extendingAction}
                \begin{tikzcd}
                    \Group \times \Gbar \times S \arrow[r, hook]\arrow[d,
                    "\id_{\Group} \times f"] & \Gbar \times \Gbar \times S \arrow[r, "\mu \times
                    \id_S"]\arrow[d, "\id_{\Gbar} \times f"] &
                    \Gbar \times S \arrow[d, "f"]\\
                    \Group \times Z \arrow[r, hook] & \Gbar \times Z \arrow[r, "\sigmabar_Z"] & Z
                \end{tikzcd}
            \end{equation}
            Its outer square (made using leftmost and rightmost columns) is
            commutative. But $\Group$ is dense in $\Gbar$, therefore also the
            rightmost square is commutative as $Z$ is separated. By restricting this square, we
            obtain the following
            Diagram~\eqref{eq:representingAction}, where $f_1\colon S\to Z$ is
            equal to $f(1_{\Gbar}, -)$.
            \begin{equation}\label{eq:representingAction}
                \begin{tikzcd}
                    \Gbar \times 1_{\Gbar} \times S \arrow[r, "\id"]\arrow[d,
                    "\id_{\Gbar} \times {f_1}"'] &
                    \Gbar \times S \arrow[d, "f"]\\
                    \Gbar \times Z \arrow[r, "\sigmabar_Z"] & Z.
                \end{tikzcd}
            \end{equation}
            This shows that $f = \sigmabar_Z \circ (\id_{\Gbar} \times f_1)$,
            so $f$ is a pullback of $\sigmabar\colon \Gbar \times
            Z\to Z$ via $f_1$. Moreover $f_1$ is uniquely determined by $f$,
            hence $(Z, \sigmabar_Z)$ is the universal object representing
            $\DX$.
        \end{proof}

        \begin{example}
            If $\Group = \Spec \kk[t^{\pm 1}]$, $\Gbar = \Spec \kk[t]$ and $\varX = \Spec A$, then
            $\Group$-action on $\varX$ is equivalent to a $\mathbb{Z}$-grading: $A =
            \bigoplus_{i=-\infty}^{\infty} A_i$. Then $\DX$ is represented by the
            closed subscheme $\Spec (A/A_{<0}\cdot A)$.
        \end{example}

        \begin{example}\label{ex:actionGoesRight}
            Let $\varX = \Spec \kk[x, y]/xy$ and consider the
            $\Gmult$-action on $\mathbb{A}^2 = \Spec \kk[x, y]$ given by
            weights $(-1, 1)$. Let $\Gbar = \mathbb{A}^1$. Then $\varX$ is $\Gmult$-stable and $\Xplus
            =\Spec \kk[y]$ is smooth, even though $\varX$ is
            not smooth.
        \end{example}
        \begin{example}\label{ex:actionGoesWrong}
            Let $\varX = \Spec \kk[x, y, z]/(xy - z^2)$. Consider the
            $\Gmult$-action on $\mathbb{A}^3 = \Spec \kk[x, y, z]$ given by
            weights $(-1, 1, 0)$. Let $\Gbar = \mathbb{A}^1$. Then $\varX$ is a $\Gmult$-stable and
            $\varX^{\Gmult} = \kk[z]/z^2$, $\Xplus = \Spec \kk[y, z]/z^2$ are
            both-nonreduced, while $\iinftyX$ is smooth.
        \end{example}

        \section{Properties of \BBname{} functors}\label{sec:BBproperties}

        In this section $\Gbar$ is linearly reductive monoid with
        connected unit group
        $\Group$ and with zero and $\varX$ is a $\Group$-scheme. Recall the
        definition of the \BBname{} functor from~\eqref{eq:BBdefinition}.
        We have already proven in Proposition~\ref{ref:representabilityForAffine:prop} that $\DX$ is
        representable for $\varX$ affine; in fact it is a
        closed subscheme. In the next section will prove that $\DX$ is
        representable for a wider class of schemes. For this we need to investigate the properties of
        $\DX$ and this is the aim of the present section.

        First, we investigate how do \BBname{} functors behave with respect to
        $\Group$-equivariant open and closed immersions.
        The intuitive idea is simple (and correct):
        \begin{itemize}
            \item if $U \subset \varX$ is a
                $\Group$-stable open subscheme and $x\in \varX(\kk)$, then $\Gbar x
                \subset U$ if and only if $0_{\Gbar} \cdot x\in U$.
            \item if $Z \subset
                \varX$ is a $\Group$-stable closed subscheme and $x\in
                \varX(\kk)$, then $\Gbar x \subset Z$
                if and only if $x\in Z(\kk)$.
        \end{itemize}
        We now justify the above claims.
        \begin{proposition}\label{ref:neighbourhoodsofzero:prop}
            Let $\kk \subset L$ be a field and $o_L$ be a $L$-rational
            point of $\Gbar_L=\Gbar\times_{\Spec k}\Spec L$ lying over
            $0_{\Gbar}\in \Gbar$. Then $\Gbar_L$ is the unique open
            $\Group_{L}$-stable subset of $\Gbar_L$ that contains $o_L$.
        \end{proposition}
        \begin{proof}
            By Lemma~\ref{ref:zerobasechange:lem} the point $o_L$ is the
            intersection of all non-empty closed $\Group_L$-stable subsets of
            $\Gbar_L$.
            Let $o_L\in U \subset \Gbar_{L}$ be an open $\Group_{L}$-stable subset of
            $\Gbar_L$. Then $Z = \Gbar_{L} \setminus U$ is closed and
            $\Group_{L}$-stable, but does not contain $o_{L}$. Hence $Z$ is
            empty, so $U = \Gbar_{L}$.
        \end{proof}

        \begin{proposition}[open immersions]\label{ref:openimmersionscartesian:prop}
            Let $i\colon U\into \varX$ be a  $\Group$-stable open
            subscheme. Then the diagram
            \[
                \begin{tikzcd}
                    \Dfunctor{U} \arrow[r, "\iinfty{U}"]\arrow[d,
                    "\Dfunctor{i}"'] &
                    U^{\Group} \arrow[d, hook, "i_{|U^{\Group}}"] \\
                    \Dfunctor{X} \arrow[r, "\iinftyX"] &
                    \varX^{\Group}
                \end{tikzcd}
            \]
            is cartesian.
        \end{proposition}
        \begin{proof}
            Consider a scheme $T$ and an element of $\Dfunctor{X}(T) \times_{\varX^{\Group}(T)}
            U^{\Group}(T)$. It corresponds to a morphism
            $f\colon\Gbar \times T\to \varX$ such that $f(0_{\Gbar} \times T) \subset
            U$. The preimage $f^{-1}(U)$ is open,
            $\Group$-stable and contains $0_{\Gbar} \times S$, hence by
            Proposition~\ref{ref:neighbourhoodsofzero:prop} it is equal to
            $\Gbar \times S$. Thus, the above $f$ induces a unique equivariant family $\Gbar
            \times T\to U$, hence an element of $\Dfunctor{U}(T)$.
        \end{proof}

        \begin{proposition}[open immersions II]\label{ref:openimmersions:prop}
            The following assertions hold.
            \begin{enumerate}
                \item\label{it:openone}  Let $i\colon U\into \varX$ be a  $\Group$-stable open
                    subscheme. Then the natural map $\Dfunctor{U}\to
                    \Dfunctor{X}$ is an open immersion.
                \item\label{it:opentwo} For every family $\cU$ of open
                    $\Group$-stable subschemes of $X$ such that
                    \[ X^{\Group}\subseteq \bigcup_{U\in \cU}U \]
                    the corresponding family of open immersions
                    \[
                        \big\{\Dfunctor{U}\to \Dfunctor{X}\big\}_{U\in \cU}
                    \]
                    is an open cover.
                \item\label{it:openthree} If $\varX$ is covered by open affine
                    $\Group$-stable
                    subschemes, then $\iinftyX\colon \DX\to \varX^{\Group}$ is
                    affine. In particular, $\DX$ is represented by a scheme
                    affine over $\varX^{\Group}$.
            \end{enumerate}
        \end{proposition}

        \begin{proof}
            Point~\eqref{it:openone}. This follows from
            Proposition~\ref{ref:openimmersionscartesian:prop}, as
            $\Dfunctor{U}\to \Dfunctor{\varX}$ is a pullback of the open immersion
            $U^{\Group} \into \varX^{\Group}$.
            Point~\eqref{it:opentwo}. This also follows from
            Proposition~\ref{ref:openimmersionscartesian:prop}, as $\coprod
            \Dfunctor{U}$ is the pullback of the open cover
            $\coprod_{U\in \cU} U^{\Group} \to \varX^{\Group}$.
            Point~\eqref{it:openthree}. Let $U$ be an affine open $\Group$-stable
            subscheme of $\varX$. By
            Proposition~\ref{ref:openimmersionscartesian:prop} we have
            $\iinftyX^{-1}(U^{\Group}) = \Dfunctor{U}$. By
            Proposition~\ref{ref:representabilityForAffine:prop},
            $\Dfunctor{U}$ is represented by an affine scheme, hence the
            preimage of affine subscheme $U^{\Group}$ under $\iinftyX$ is
            affine. Since $\varX$ is covered by such $U$, we see that
            $\iinftyX$ is affine.
        \end{proof}
        \begin{remark}
            Proposition~\ref{ref:openimmersions:prop} together with
            Proposition~\ref{ref:representabilityForAffine:prop} and
            Theorem~\ref{ref:etalestableneighbourhoods:thm} show that under
            mild assumptions the map
            $\ioneX\colon \DX \to \varX$ is \'etale-locally a closed
            immersion. This is not always true Zariski-locally, as the case of
            nodal curve with $\Gmult$-action shows.
        \end{remark}

        \begin{theorem}[closed immersions]\label{ref:closedImmersions:prop}
            Let $j:Z\to X$ be a $\Group$-equivariant closed immersion. Then the square of functors
            \begin{center}
                \begin{tikzpicture}
                    [description/.style={fill=white,inner sep=2pt}]
                    \matrix (m) [matrix of math nodes, row sep=3em, column sep=2em,text height=1.5ex, text depth=0.25ex]
                    { \Dfunctor{Z} &  &     Z              \\
                    \Dfunctor{X} &   &                    X    \\} ;
                    \path[->,font=\scriptsize]
                    (m-1-1) edge node[above] {$\ione{Z} $} (m-1-3)
                    (m-1-1) edge node[left] {$\Dfunctor{j}$} (m-2-1)
                    (m-1-3) edge node[auto] {$j$} (m-2-3)
                    (m-2-1) edge node[below] {$\ione{X}$} (m-2-3);
                \end{tikzpicture}
            \end{center}
            is cartesian. In particular, $\Dfunctor{j}:\Dfunctor{Z}\to \Dfunctor{X}$ is a closed immersion of functors.
        \end{theorem}

        \begin{proof}
            Fix a $\kk$-scheme $T$ with morphisms $T\to Z$ and $T\to
            \Dfunctor{X}$ agreeing on $\varX$. This data amount to an
            equivariant family $f\colon \Gbar \times T\to \varX$ such that
            $f(1_{\Group} \times T) \subset Z$. Since $Z$ is
            $\Group$-stable, we have $f(\Group \times T) \subset Z$. Since $Z$ is
            closed, we have $f(\Gbar \times T) \subset Z$, hence a family $f\colon
            \Gbar \times T\to Z$, corresponding to an element of
            $\Dfunctor{Z}(T)$.
        \end{proof}

        We now prove that the \BBname{} decompositions are exactly those
        $\Group$-schemes for which the $\Group$-action extends to a
        $\Gbar$-action. For separated $\Group$-schemes this is fairly trivial.
        For non-separated it requires a little work and the lemma below.

        \begin{lemma}\label{ref:isomorphismGbarAgreementTopological:lem}
            Let $\varX$ be a $\Gbar$-scheme and
            $\varphi_1, \varphi_2\colon\Gbar \to \varX$ be two
            $\Group$-equivariant morphisms. If $\varphi_1$ and $\varphi_2$
            agree on $\Group$, then $\varphi_1 = \varphi_2$.
        \end{lemma}
        \begin{proof}
            By assumption of this section $\Gbar$ has a zero. To lighten notation, let $0 := 0_{\Gbar}$.
            Let $y_i = \varphi_i(0)$ for $i=1,2$.
            Let $\pi = 0 \cdot (-)\colon \varX \to \varX^{\Group}$ be
            the multiplication by zero. Note that $j\colon \varX^{\Group} \into \varX$
            is a section of $\pi$.
            Let $x_0 = \pi(\varphi_1(1)) =
            \pi(\varphi_2(1))$. This is a $\kk$-point of a $\kk$-scheme
            $\varX^{\Group}$, hence it is closed. Thus, for $i=1,2$, the
            fiber $(\pi \circ \varphi_i)^{-1}(x_0) \subset
            \Gbar$ is closed, $\Group$-stable, and contains $1_{\Group}$,
            thus $(\pi \circ \varphi_i)^{-1}(x_0) = \Gbar$. In particular,
            $\pi(y_i) = x_0$ for $i=1,2$, so $y_i = (j\circ \pi)(y_i) =
            j(x_0)$ for $i=1,2$. Hence, $y_1 = y_2$.

            Recall that $\Gbar$ is geometrically integral.
            Consider the equalizer $E \subset \Gbar$ of
            $\varphi_1$, $\varphi_2$. It is a locally closed subset. Since $E$
            contains $\Group$, it is dense. Thus $E$ is open. Since $\varphi_i$
            are $\Group$-equivariant, $E$ is $\Group$-stable.
            Since $\varphi_1(0) = \varphi_2(0)$, we have
            $0\in E$. By
            Proposition~\ref{ref:neighbourhoodsofzero:prop}, we have
            $E = \Gbar$, thus $\varphi_1 = \varphi_2$.
        \end{proof}

        \begin{proposition}[BB decomposition for $\Gbar$-scheme $\varX$]\label{ref:isomorphism:prop}
            Let $\varX$ be a $\Group$-scheme. Then $\ioneX\colon
            \Dfunctor{X}\to \varX$ is an isomorphism if and only if the
            $\Group$-action on $\varX$ extends to a $\Gbar$-action.
        \end{proposition}
        \begin{proof}
            The $\Group$-action on $\Dfunctor{X}$ extends to a $\Gbar$-action,
            hence ``only if'' is proved. Suppose that ``if'' holds, i.e., that
            $\varX$ is a $\Gbar$-scheme.
            For a scheme $S$ and $\psi\colon S\to \varX$ we have a unique
            $\Gbar$-equivariant family $\psi'\colon \Gbar \times S\to \varX$
            given by $\psi'(\bar{g}, s) = \bar{g}\psi(s)$.
            The transformation $j\colon\psi\mapsto \psi'$ is the inverse of $\ioneX$.
            Indeed, $\ioneX \circ j = \id$ obviously. Choose any $S$ and a $\Group$-equivariant family $\varphi\colon
            \Gbar \times S\to \varX$. Then $\varphi' = (j \circ
            \ioneX)(\varphi)\colon \Gbar \times S\to \varX$ is a
            $\Group$-equivariant family which agrees with $\varphi$ on $\Group \times S$.
            (if $\varX$ was separated, then we could immediately conclude that
            $\varphi = \varphi'$. The remaining part deals with non-separated
            $\varX$).

            We now prove that $\varphi = \varphi'$.
            Let $E \into \Gbar \times S$ be the equalizer of $\varphi$ and
            $\varphi'$. This is a $\Group$-stable locally closed subscheme, so there exists an
            open $U \subset \Gbar \times S$ and a closed immersion $E\into U$.
            We have $\Group \times S \subset E$. The subset $\Group \times S$
            is schematically-dense in $\Gbar \times S$, hence also in $U$.
            Thus $E$ is also schematically dense in $U$, but it is closed,
            hence $E = U$.

            It remains to prove $E = U$ contains all points of $\Gbar \times S$.
            For this choose a point $p\in \Gbar \times S$ and let $K$ be its
            residue field. After a base change via $\Spec(K)\to \Spec(\kk)$
            we may and will assume that $K = \kk$.
            Then $p = (p_1, p_2)$, where $p_i$ are $\kk$-points of $\Gbar$ and
            $S$ respectively.
            Consider the restrictions of $\varphi$ and $\varphi'$ to $\Gbar \times p_2$.
            By Lemma~\ref{ref:isomorphismGbarAgreementTopological:lem} they
            are equal. Thus $\Gbar \times p_2 \subset E$ and in particular
            $p = (p_1, p_2)\in E$. This shows that $E = \Gbar \times S$, so
            $\varphi = \varphi'$.
        \end{proof}

        Next, we note that $\ioneX$ is a monomorphism for separated $\varX$.
        Later in Proposition~\ref{ref:ioneLocallyClosed:prop} we will see that
        when $\varX$ is normal $\ioneX$ is a closed immersion when restricted
        to each irreducible component. The classical \BBname{} decomposition of
        $\mathbb{P}^1$ with standard $\Gmult$ action shows that one cannot
        hope that $\ioneX$ is an immersion globally. The question whether $\ioneX$ is
        an immersion on each \emph{connected} component is quite subtle,
        see~\cite[Appendix~B]{Drinfeld} for some counter-examples.
        \begin{lemma}\label{ref:monomorphism:lem}
            Let $\varX$ be a separated $\Group$-scheme. For every $\kk$-scheme $T$, the map $\ione{X}(T)\colon
            \Dfunctor{X}(T)\to \Mor(T, X)$ is an injection. Hence $\ione{X}$
            is a monomorphism of functors.
        \end{lemma}
        \begin{proof}
            Let $\varphi\colon \Gbar \times T\to \varX$ be a
            $\Group$-equivariant map and $\alpha\colon 1_{\Group} \times T\to
            \varX$ be its restriction. The restriction $\varphi' = \varphi_{\mid \Group \times
            T}\colon \Group \times T\to \varX$ is $\Group$-equivariant, hence
            it factors as
            \[
                \begin{tikzcd}
                    \varphi'\colon \Group \times T \arrow[r, "\id \times
                    \alpha"] & \Group \times \varX \to \varX
                \end{tikzcd}
            \]
            and so it is uniquely determined by $\alpha$. Since $\Group \times
            T$ is schematically dense in $\Gbar \times T$ and $\varX$ is separated, also
            $\varphi$ is uniquely determined by $\alpha$.
        \end{proof}

        \begin{proposition}\label{ref:etaleFormally:prop}
Let $f:X\to Y$ be a $\Group$-equivariant and formally {\'e}tale (resp.~unramified) morphism. Then $\Dfunctor{X}\to \Dfunctor{Y}$ is a formally {\'e}tale (resp.~unramified) morphism of functors.
\end{proposition}
\begin{proof}
We use the infinitesimal lifting criterion. Consider an affine scheme $Z$ over
$\kk$ and an ideal $I\subseteq H^0(Z,\cO_Z)$ such that $I^2=0$. Let
$i:Z_I\to Z$ be the closed immersion corresponding to the ideal $I$.  Consider
$\Group$-equivariant morphisms
\[
    \left(\varphi:\Gbar\times Z\to Y\right)\in
    \Dfunctor{Y}(Z),\quad\left(\eta:\Gbar\times Z_I\to X\right)\in
    \Dfunctor{X}(Z_I),
\]
such that the Diagram~\eqref{eq:etaleness} is commutative
\begin{equation}\label{eq:etaleness}
\begin{tikzpicture}
[description/.style={fill=white,inner sep=2pt}]
\matrix (m) [matrix of math nodes, row sep=3em, column sep=2em,text height=1.5ex, text depth=0.25ex]
{ \Gbar\times Z_I &  &    X              \\
 \Gbar\times Z &   &        Y               \\} ;
\path[->,font=\scriptsize]
(m-1-1) edge node[above] {$ \eta$} (m-1-3)
(m-1-1) edge node[left] {$1_{\Gbar}\times i$} (m-2-1)
(m-1-3) edge node[auto] {$f$} (m-2-3)
(m-2-1) edge node[below] {$\varphi$} (m-2-3);
\path[densely dotted,->,font=\scriptsize]
(m-2-1) edge node[below] {$\tau$} (m-1-3);
\end{tikzpicture}
\end{equation}
of $\kk$-schemes.
Clearly, $\Gbar \times Z_I \subset \Gbar \times Z$ is a closed immersion given
by the ideal $H^0(\Gbar,\cO_{\Gbar})\otimes I$ which square is equal to
zero. If $f$ is unramified, there exists at most one lift $\tau\colon \Gbar\times
Z\to X$, which ends the proof in this case. Assume that $f$ is {\'e}tale, so
that
there does exist a unique lift $\tau\colon \Gbar\times
Z\to X$ in Diagram~\eqref{eq:etaleness}. We now prove that $\tau$ is
$\Group$-equivariant. Then it gives an element of $\Dfunctor{X}(Z)$, which
restricts to $\varphi$ and $\eta$, thus proving \`etaleness of $\Dfunctor{X}\to
\Dfunctor{Y}$.

To prove that $\tau$ is equivariant, denote by $a_{S}\colon \Group \times S\to
S$ the action of $\Group$ on $S\in \{X, Y, \Gbar\}$ and consider a commutative square
\begin{center}
\begin{tikzpicture}
[description/.style={fill=white,inner sep=2pt}]
\matrix (m) [matrix of math nodes, row sep=3em, column sep=2em,text height=1.5ex, text depth=0.25ex] 
{ \Group\times \Gbar\times Z_I & &  & &&&   X              \\
&&& &&&\\
 \Group\times \Gbar\times Z &   & &   &&&    Y               \\} ;
\path[->,font=\scriptsize]  
(m-1-1) edge node[above] {$ a_X\cdot\left(1_{\Group}\times \eta\right)=\eta\cdot \left(a_{\Gbar}\times 1_{Z_I}\right)$} (m-1-7)
(m-1-1) edge node[left] {$1_{\Group}\times 1_{\Gbar}\times i$} (m-3-1)
(m-1-7) edge node[auto] {$f$} (m-3-7)
(m-3-1) edge node[below] {$a_Y\cdot\left(1_{\Group}\times \varphi\right)=\varphi \cdot \left(a_{\Gbar}\times 1_{Z}\right)$} (m-3-7);
\path[densely dotted,->,font=\scriptsize]
(m-3-1) edge[transform canvas={yshift=0.2ex}] node[above=24pt, left=-50pt] {$a_X\cdot \left(1_{\Group}\times \tau\right)$} (m-1-7)
(m-3-1) edge[transform canvas={yshift=-0.2ex}] node[below=24pt, right=-50pt] {$\tau\cdot \left(a_{\Gbar}\times 1_Z\right)$} (m-1-7);
\end{tikzpicture}
\end{center}
Both $a_X\cdot \left(1_{\Group}\times \tau\right)$ and $\tau\cdot
\left(a_{\Gbar}\times 1_Z\right)$ are lifts along formally {\'e}tale morphism
$f$. Thus they are equal, which exactly means that $\tau$ is
$\Group$-equivariant.
\end{proof}

For \'etale (not only formally \'etale) morphisms $f$ we have a more precise
description of $\Hfunctor{f}$ defined in the introduction, see \eqref{eq:definitionFormal}.
First, we formulate a technical lemma.
\begin{lemma}\label{ref:fixedpointsandetale:lem}
    Let $\Group$ be a geometrically connected algebraic group.
    Let $\varX$ be a $\Group$-scheme, $Y$ be an algebraic space with the trivial
    $\Group$-action and $X\to Y$ be an \'etale $\Group$-equivariant morphism.
    Then the $\Group$-action on $\varX$ is trivial.
\end{lemma}
\begin{proof}
    Choose an \'etale, surjective morphism $Y'\to Y$ from a scheme $Y'$. Equip
    $Y'$ with the trivial $\Group$-action.  Then $X' = \varX \times_{Y} Y'$ is a
    scheme with a natural $\Group$-action and $f'\colon X'\to Y'$ is an \'etale
    $\Group$-equivariant morphism of schemes.

    First, we prove that the $\Group$-action on the set $|X'|$ is trivial. Take any $x\in X'$ and let
    $y = f'(x)\in Y^{\Group}$. Since $f'$ is \'etale, the set $f'^{-1}(y)$ is
    discrete, hence the connected group $\Group$ acts trivially on this set,
    in particular $x\in |X'|^{\Group}$.
    Second, we prove that the scheme-theoretic action is trivial. Consider the
    action map $\sigma\colon \Group \times X' \to X'$ and the
    projection $pr_2\colon \Group \times X'\to X'$. As the action on
    $Y'$ is trivial, we have $f' \circ \sigma =
    f' \circ pr_2$. Therefore, we obtain a map
    \[
        (\sigma, pr_2)\colon \Group \times X'\to X' \times_{Y'} X'.
    \]
    We claim that it factors through the diagonal $\Delta\colon X'\to X'
    \times_{Y'} X'$. The morphism $f'$ is \'etale, hence $\Delta$ is
    an open immersion, so it is enough to check that the factorization exists
    set-theoretically. This follows, because $\Group$ acts trivially on
    $|X'|$. Hence $pr_2 = \sigma$, which concludes the proof that $X' =
    X'^{\Group}$.
    Now, $X'\to X$ is \'etale surjective and $\Group$-equivariant, so also the
    $\Group$-action on $X$ is trivial.
\end{proof}

\begin{theorem}\label{ref:etale:thm}
    Let $\varX$ be a $\Group$-scheme, $Y$ be a $\Group$-algebraic space and
    $f\colon\varX\to Y$ be an \'etale equivariant morphism. Then we have a diagram consisting of two cartesian
    squares
    \[
        \begin{tikzcd}
            \Hfunctor{X} \arrow[r, "\iinfty{X}"]\arrow[d, "\Hfunctor{f}"] & X^{\Group} \arrow[r]\arrow[d] & X\arrow[d, "f"]\\
            \Hfunctor{Y} \arrow[r, "\iinfty{Y}"] & Y^{\Group} \arrow[r] & Y
        \end{tikzcd}
    \]
    In particular, $\Hfunctor{f}\colon \Hfunctor{X}\to \Hfunctor{Y}$ is a
    pullback of $f$, so $\Hfunctor{f}$ is \'etale.
\end{theorem}
\begin{proof}
    By assumption (see the beginning of the present section), $\Group$ is
    connected.
    Consider first the closed subscheme $X' = Y^{\Group} \times_Y X \into
    X$, as in
    \begin{equation}\label{eq:pullbackOne}
        \begin{tikzcd}
            X' \arrow[r]\arrow[d, "f'"] & X\arrow[d, "f"]\\
            Y^{\Group} \arrow[r] & Y
        \end{tikzcd}
    \end{equation}
    We claim that $X' = X^{\Group}$.  Clearly, $X^{\Group} \subset X'$, so
    it is enough to prove that the $\Group$ action on $X'$ is trivial. This
    follows from Lemma~\ref{ref:fixedpointsandetale:lem}.
    Let $\iinfty{Y}\colon \Hfunctor{Y}\to Y^{\Group}$ be the natural map and
    complete the cartesian diagram~\eqref{eq:pullbackOne} to
    obtain~\eqref{eq:pullbackTwo}.
    \begin{equation}\label{eq:pullbackTwo}
        \begin{tikzcd}
            X'' \arrow[r]\arrow[d] & X^{\Group} \arrow[r]\arrow[d, "f'"] & X\arrow[d, "f"]\\
            \Hfunctor{Y} \arrow[r] & Y^{\Group} \arrow[r] & Y
        \end{tikzcd}
    \end{equation}
    We claim that $X''$ is isomorphic to $\Hfunctor{X}$.
    Indeed, $X''(S) = \Hfunctor{Y}(S) \times_{Y^{\Group}(S)} X^{\Group}(S)$.
    In plain terms, $\Hfunctor{Y}(S)$ is a coherent sequence of families
    $\varphi_n\colon \Gbar_n \times S\to Y$, $X^{\Group}(S)$ is a family
    $\delta_0\colon\Gbar_0 \times S\to \varX$ and we have $\varphi_0 =
    f' \circ \delta_0$. We now construct a (unique) coherent family
    $\delta_n\colon \Gbar_n \times S\to \varX$. We already have $\delta_0$.
    Suppose that we constructed $\delta_n$. Consider Diagram~\eqref{eq:lifting}.
    \begin{equation}\label{eq:lifting}
        \begin{tikzcd}
            \Gbar_n \times S \arrow[d, hook]\arrow[r, "\delta_{n}"] &\varX
            \arrow[d, "f"]\\
            \Gbar_{n+1} \times S \arrow[r, "\varphi_{n+1}"] & Y.
        \end{tikzcd}
    \end{equation}
    Since $f$ is \'etale, there exists a unique $\delta_{n+1}\colon
    \Gbar_{n+1} \times S\to \varX$ preserving commutativity
    of Diagram~\eqref{eq:lifting}. As in the proof
    of Proposition~\ref{ref:etaleFormally:prop}, uniqueness implies that
    $\delta_{n+1}$ is $\Group$-equivariant.
\end{proof}

    \section{Formal \BBname{} functors and representability}\label{sec:algebraization}

        In this section we prove Theorem~\ref{ref:algebraizationABB}.
        As before, we assume that $\Gbar$ is a
        linearly reductive monoid with zero and with connected unit group
        $\Group$, and that $\Gbar_n$ is the $n$-th infinitesimal
        neighbourhood of $0\in \Gbar$.
        As above, $\varX$ is a $\Group$-scheme.

        A $\Group$-equivariant morphism $\Gbar \times Z\to \varX$
        restricts to $\Gbar_n\times Z\to \varX$ for all $n$. Hence, we obtain a restriction
        $\Dfunctor{X}\to \Hfunctor{X}$ and
        Theorem~\ref{ref:algebraizationABB} asserts that it is an
        isomorphism of sheaves represented by schemes.
        The idea is to first prove the representability of $\Hfunctor{X}$ and then
        deduce the isomorphism of $\Dfunctor{X}$ and $\Hfunctor{X}$.

        The functor $\Hfunctor{X}$ depends only on the formal neighbourhood of
        $\varX^{\Group} \subset \varX$. Therefore below we consider formal
        $\Group$-schemes which abstract the properties of formal completion of
        $\varX$ along $\varX^{\Group}$.

        \newcommand{\Zhat}{\widehat{Z}}%
        \newcommand{\What}{\widehat{W}}%
        \newcommand{\Zformal}{\cZ}%
        \newcommand{\Wformal}{\cW}%
        \begin{definition}
            A \emph{formal $\Group$-scheme} is a sequence $\Zformal =
            \{Z_n\}_{n\in \NN}$ of $\Group$-schemes together with equivariant closed immersions
            \[
                \begin{tikzcd}
                    Z_0 \arrow[r, hook] & Z_1 \arrow[r, hook] &  \ldots
                    \arrow[r, hook] & Z_n \arrow[r, hook] & \ldots
                \end{tikzcd}
            \]
            \begin{enumerate}
                \item For every $n\in \NN$ we have $Z_0=Z_n^{\Group}$
                    scheme-theoretically,
                \item Let $\cI_n$ be an ideal of $\cO_{Z_n}$ defining $Z_0$.
                    Then for every $m\leq n$ the subscheme
                    $Z_m \subset Z_n$ is defined by
                    $\cI_n^{m+1}$.
            \end{enumerate}
            If each $Z_n$ is a $\Gbar$-scheme, then we say that $\Zformal$ is a
            \emph{formal $\Gbar$-scheme}. If each $Z_n$ is (locally)
            Noetherian, we say that \emph{$\Zformal$ is (locally) Noetherian}.
        \end{definition}

        \begin{example}[Formal schemes from algebraic ones]\label{algebraicformalgroupschemes}
            Let $Z$ be a $\Group$-scheme and $\cI$ be the ideal of
            $Z^{\Group}$. Then $Z_n=V(\cI^{n+1})$ is a closed $\Group$-stable
            subscheme of $Z$ for every $n\in \NN$ and this yields to a formal
            $\Group$-scheme $\Zformal=\{Z_n\}_{n\in \NN}$. We denote this formal
            $\Group$-scheme by $\Zhat$.
        \end{example}

The data of a formal $\Gbar$-scheme can be respelled algebraically, as
follows. Let $\Zformal = \{Z_n\}$ be a formal $\Gbar$-scheme locally of
finite type and fix $n\in \NN$.  The schemes $Z_0, Z_n \subset Z_{n+1}$ are
given by the ideals $\cI_{n+1}$ and $\cI_{n+1}^{n}$. Therefore, the ideal of 
$Z_0 \subset Z_n$ is nilpotent. In particular, points of $Z_0$ and $Z_n$ are the same. Hence, the $\Gbar$ action on $Z_n$ stabilizes all its open subschemes.  Since any affine cover of $Z_n$ is $\Gbar$-stable, the scheme $Z_n$ represents its \BBname{}
functor. Hence, there exists a retraction $\iinfty{Z_n}\colon Z_n \to
Z_n^{\Group} = Z_0$ with section $Z_0 \into Z_n$. As $Z_0 \subset Z_n$ is
given by a nilpotent ideal, the morphism $\iinfty{Z_n}$ is finite. Hence $Z_n
\simeq \Spec_{Z_0} \cA_n$, where
\begin{equation}\label{eq:Andefinition}
    \cA_n = (\iinfty{Z_n})_*\cO_{Z_n}
\end{equation}
is a sheaf of quasi-coherent
$\cO_{Z_0}$-algebras. Moreover, $\cA_n$ is a sheaf of $\Gbar$-algebras. The
$\Gbar$-action on each $Z_n$ is
topologically trivial, hence translates into the $\Gbar$-action on the
algebra $H^0(U, \cA_n)$ for each open $U$.
To avoid confusion, we formalize the structure on $\cA_n$.
\begin{definition}
    Let $Z$ be a $\Gbar$-scheme, such that $\Gbar$ acts trivially on the
    topological space $|Z|$. A \emph{quasi-coherent sheaf of $\Gbar$-modules
    (resp. $\Gbar$-algebras)} is a quasi-coherent sheaf $\cA$ of
    $\cO_{Z}$-modules (resp. $\cO_{Z}$-algebras) together with a $\Gbar$-action on
    each $H^0(U, \cA)$, for $U \subset Z$ open, which is $\cO_{Z}$-linear
    (resp.~by $\cO_Z$-linear automorphisms).
\end{definition}

The closed immersions $Z_{n} \into Z_{n+1}$ commute with projections to $Z_0$
and hence induce $\Gbar$-equivariant surjections $\cA_{n+1}
\onto \cA_{n}$ of $\Gbar$-algebras.

        Now we define morphisms of formal $\Group$-schemes.
        \begin{definition}
            Let $\Zformal =\{Z_n\}$ and $\Wformal =\{W_n\}$ be formal $\Group$-schemes.
            A \emph{morphism $\varphi:\Wformal\to \Zformal$ of formal
            $\Group$-schemes} is a family of $\Group$-equivariant morphisms
            $\varphi=\{\varphi_n:W_n\to Z_n\}$ such that for every $n\in \NN$
            we have a commutative square
            \begin{center}
                \begin{tikzpicture}
                    [description/.style={fill=white,inner sep=2pt}]
                    \matrix (m) [matrix of math nodes, row sep=3em, column sep=2em,text height=1.5ex, text depth=0.25ex] 
                    {W_{n+1}&  &    Z_{n+1}          \\
                W_n&   & Z_n         \\} ;
                \path[->,font=\scriptsize]  
                (m-1-1) edge node[above] {$\varphi_{n+1}$} (m-1-3)
                (m-2-1) edge node[below] {$\varphi_n$} (m-2-3);
                \path[right hook->,font=\scriptsize]
                (m-2-3) edge node[right] {$ $} (m-1-3)
                (m-2-1) edge node[right] {$ $} (m-1-1);
            \end{tikzpicture}
        \end{center}
    \end{definition}
    \begin{remark}[Morphisms of formal $\Gbar$-schemes are $\Gbar$-equivariant]
        Let $\Wformal$ and $\Zformal$ be formal $\Gbar$-schemes and consider
        their morphism $\varphi:\Wformal \to \Zformal$ (as formal
        $\Group$-schemes). Then for every $n\in \NN$ the morphism
        $\varphi_n\colon W_n\to Z_n$ is $\Gbar$-equivariant. To see this,
        consider Diagram~\eqref{eq:Gbarformaldiagram}.
        \begin{equation}\label{eq:Gbarformaldiagram}
            \begin{tikzpicture}
                [description/.style={fill=white,inner sep=2pt}]
                \matrix (m) [matrix of math nodes, row sep=3em, column sep=2em,text height=1.5ex, text depth=0.25ex]
                {W_{n}&  &  W_0\times_{Z_0}Z_n  & &Z_n          \\
                     &   & W_0           & &Z_0     \\} ;
                \path[->,font=\scriptsize] 
                (m-1-1) edge[bend left = 30] node[above] {$ \varphi_n $} (m-1-5) 
                (m-1-1) edge node[below] {$r_n $} (m-1-3)
                (m-1-3) edge node[right] {$p_n $} (m-2-3)
                (m-1-1) edge node[left = 3pt, below = 3pt] {$\iinfty{W_n} $} (m-2-3)
                (m-1-3) edge node[below] {$q_n $} (m-1-5)
                (m-2-3) edge node[below] {$\varphi_0 $} (m-2-5)
                (m-1-5) edge node[right] {$\iinfty{Z_n} $} (m-2-5);
            \end{tikzpicture}
        \end{equation}
        Since $W_0$ and $Z_0$ are equipped with trivial $\Gbar$-actions, also the
        pullback $W_0\times_{Z_0}Z_n$ is a $\Gbar$-scheme and $q_n$ is
        $\Gbar$-equivariant. Recall that $\iinfty{Z_n}$, $\iinfty{W_n}$ are
        affine morphisms. Therefore, $p_n$ is affine. Hence $r_n$ is a
        $\Group$-equivariant morphism between $\Gbar$-schemes separated (even
        affine) over $W_0$. Thus $r_n$ is $\Gbar$-equivariant.
    \end{remark}

    \begin{definition}\label{ref:locallylinear:def}
        A \emph{locally linear $\Gbar$-scheme} is a $\Gbar$-scheme which
        admits an open cover by affine $\Gbar$-stable subschemes. The category
        of locally linear $\Gbar$-schemes consists of those schemes and
        $\Gbar$-equivariant morphisms.
    \end{definition}
    Let $Z$ be a locally linear $\Gbar$-scheme. By Proposition~\ref{ref:isomorphism:prop}, the
    map $\Dfunctor{Z}\to Z$ is an isomorphism. In particular, there is a
    canonical morphism $\iinfty{Z}\colon Z\to Z^{\Group}$, which is the
    multiplication by zero. For an affine open $\Gbar$-stable cover
    $\{V_i\}_i$
    of $Z$, we have $V_i = \iinfty{Z}^{-1}(\iinfty{Z}(V_i))$ by
    Proposition~\ref{ref:openimmersionscartesian:prop}, hence the canonical morphism $\iinfty{Z}\colon Z\to Z^{\Group}$ is affine.

    \begin{definition}\label{ref:algebraization:def}
        Let $\Zformal$ be a formal $\Gbar$-scheme. An \emph{algebraization} of
        $\Zformal$ is a $\Gbar$-scheme $Z$ such that
        \begin{enumerate}
            \item\label{it:algebraization:one} $Z$ is a locally linear
                $\Gbar$-scheme.
            \item\label{it:algebraization:two} $\Zformal$ and $\Zhat$ are isomorphic formal $\Gbar$-schemes.
        \end{enumerate}
    \end{definition}
    By the above discussion, the morphism
    $\iinfty{Z}\colon Z\to Z^{\Group}$ is affine for any algebraization $Z$.

    \begin{theorem}[Algebraization of a formal $\Gbar$-scheme]\label{ref:algebraizationOfFormalSchemes:thm}
        Let $\Zformal =\{Z_n\}$ be a formal $\Gbar$-scheme. Then there exists
        a colimit
        \[
            Z = \colim_{n} Z_n
        \]
        in the category of locally linear $\Gbar$-schemes and $Z$ is the
        unique algebraization of $\Zformal$.
        If in addition $\Zformal$ is locally Noetherian, then $\iinfty{Z}$ is of finite type. If
        $\Zformal$ is locally Noetherian and $Z_0$ is of finite type, then also $Z$ is of
        finite type.
    \end{theorem}

Now we spell out the main idea of the proof: the $\Gbar$-scheme $Z$
required in Theorem~\ref{ref:algebraizationOfFormalSchemes:thm} is equal to $\Spec_{Z_0} \cA$, where
$\cA$ is the limit of $\cA_n$ \emph{in the category of
$\Gbar$-algebras}; in other words each isotypic component of $\cA$ is the
limit of isotypic components of $\cA_n$.
Our first goal is to prove a stabilization result.
We denote by $\Irr(\Group)$ the set of isomorphism types of irreducible
$\Group$-representations and by $\Irr(\Gbar) \subset \Irr(\Group)$ the
subset of $\Gbar$-representations. For $\lambda\in \Irr(\Group)$ and
a quasi-coherent $\Gbar$-module $\cC$ on $Z_0$ we denote by $\cC[\lambda]
\subset \cC$ the $\Gbar$-submodule such that $H^0(U, \cC[\lambda]) \subset H^0(U, \cC)$
is the union of all $\Group$-subrepresentations of $H^0(U, \cC)$ isomorphic to
$\lambda$ (i.e., the isotypic component of $\lambda$).

\begin{lemma}[stabilization on an isotypic component]\label{ref:stability:lem}
Let $\lambda\in \Irr(\Gbar)$. Then there exists a number $n_{\lambda}\in \NN$
such that the following holds. Let $\Zformal=\{Z_n\}$ be a formal $\Gbar$-scheme
and $\{\cA_{n+1} \onto \cA_n\}$ be the associated sequence of quasi-coherent
$\Gbar$-algebras. Then for every $n > n_{\lambda}$ the surjection
\[
    \cA_{n}[\lambda] \onto \cA_{n-1}[\lambda]
\]
is an isomorphism. If $\lambda_0\in \Irr(\Gbar)$ is the
trivial representation, then we may take $n_{\lambda_0}=0$.
\end{lemma}
\begin{proof}[Proof of Lemma~\ref{ref:stability:lem}]
    The claims are preserved under field extension, so we may assume our field
    is algebraically closed (hence perfect) so we may use the Kempf's torus.
    Fix a grading on
    $\kk[\Gbar]$ induced by a Kempf's torus for $\kk$ as in
    Corollary~\ref{ref:KempfTorus:cor}. Denote by $A_{\lambda}\subseteq \NN$
    the set of weights which appear in
    $\kk[\Group]_{\lambda}$. Since $\dim_{\kk}\kk[\Group]_{\lambda}$
    is finite by Proposition~\ref{ref:isotypiccomponents:prop}, the set
    $A_{\lambda}$ is finite. Put
    \[
        n_{\lambda}=\sup A_{\lambda}.
    \]
    Fix $n> n_{\lambda}$ and let
    $\cI_n = \ker(\cA_n\to \cA_0)$. Then we have a decomposition
    with respect to the chosen torus
    \[
        \cA_n=\bigoplus_{i\geq 0}(\cA_n)[i],
    \]
    By Corollary~\ref{ref:KempfTorus:cor}, we have $\cI_n =
    \bigoplus_{i\geq 1}(\cA_n)[i]$. Since $n > n_{\lambda}$ we have
    \[
        \cI^{n}_n \subset \bigoplus_{i\geq
        n}(\cA_n)[i]\subseteq \bigoplus_{i \not \in
            A_{\lambda}}(\cA_n)[i]
    \]
Hence, $\cI^{n}_n[\lambda] = 0$. But
$\cI^{n}_n[\lambda] = \ker(\cA_{n}[\lambda] \to
\cA_{n-1}[\lambda])$, thus $\cA_{n}[\lambda] \to \cA_{n-1}[\lambda]$ is an
isomorphism.
Finally note that $A_{\lambda_0}=\{0\}$. This implies that $n_{\lambda_0}=0$.
\end{proof}

\begin{proof}[Proof of Theorem~\ref{ref:algebraizationOfFormalSchemes:thm}]
    Let $\cA_n$ be the quasi-coherent $\Gbar$-algebras as
    in~\eqref{eq:Andefinition}. For $\lambda\in \Irr(\Gbar)$ we define
    $\cA[\lambda] := \cA_n[\lambda]$, where $n\geq n_{\lambda}$ as in
    Lemma~\ref{ref:stability:lem}.
    \[
        \cA=\bigoplus_{\lambda\in
            \Irr(\Gbar)}\cA[\lambda]=\bigoplus_{\lambda\in
                \Irr(\Gbar)}\cA_{n_{\lambda}}[\lambda].
    \]
    Clearly $\cA[\lambda_0] = \cA_0 = \cO_{Z_0}$ canonically (where
    $\lambda_0$ is the trivial representation), hence $\cA$ is an
    $\cO_{Z_0}$-module.
    Actually $\cA=\lim_{n}\cA_n$ in the category of quasi-coherent
    $\Gbar$-modules on $Z_0$.
    We construct the algebra structure on $\cA$. For this
    pick $\eta_1, \eta_2\in \Irr(\Gbar)$. Fix the finite set
    $\{\lambda_1, \ldots ,\lambda_s\}\subseteq \Irr(\Gbar)$ of representations
    which appear in $\kk[\Gbar]_{\eta_1}\otimes_{\kk}\kk[\Gbar]_{\eta_2}$.
    Then, for every $n\in \NN$, we have the multiplication
$$\cA_n[\eta_1]\otimes_{\kk} \cA_n[\eta_2]\to \cA_n[\eta_1]\cdot \cA_n[\eta_2]\subseteq \bigoplus_{i=1}^s\cA_n[\lambda_i]$$
and by Lemma \ref{ref:stability:lem} these morphisms can be identified for $n\geq \sup \{n_{\eta_1},n_{\eta_2},n_{\lambda_1},...,n_{\lambda_s}\}$. We define
$$\cA[\eta_1]\otimes_{\kk} \cA[\eta_2]\to  \bigoplus_{i=1}^s\cA[\lambda_i]\subseteq \cA$$
as a morphism induced by the multiplication morphism for any $n\geq \sup
\{n_{\eta_1},n_{\eta_2},n_{\lambda_1}, \ldots ,n_{\lambda_s}\}$. This gives an
$\cO_{Z_0}$-algebra structure on $\cA$, so $\cA$ is in fact the limit of
$\cA_n$ is the category of $\Gbar$-algebras. Note that from the description of
$\cA$ it follows that for every $n\in \NN$ we have a surjective morphism
$p_n:\cA\onto \cA_n$ of $\Gbar$-algebras. We denote its kernel
by $\cJ_n$ and we put $\cJ:=\cJ_0$. The natural injection $\cO_{Z_0} = \cA_0 \to \cA$ is a section
of $p_0$, so that we have
\[
    \cJ=\bigoplus_{\lambda \in \Irr(\Gbar)\setminus
        \{\lambda_0\}}\cA[\lambda].
\]
We also denote by $\cI_n$ the kernel of $\cA_n\twoheadrightarrow
\cA_0=\cO_{Z_0}$ for $n\in \NN$. Then $\cI_n=\cJ/\cJ_n$.
Fix $m\in \NN$ and consider $n\in \NN$
such that $n\geq m$. Since $\Zformal$ is a formal $\Gbar$-scheme, the sheaf
$\cI_n^{m+1}$ is the kernel of the morphism $\cA_n\twoheadrightarrow \cA_m$.
Thus
\[
\cJ_m/\cJ_n=\cI_n^{m+1}=(\cJ^{m+1}+\cJ_n)/\cJ_n.
\]
Both $\cJ_m$ and $\cJ^{m+1}$ are $\Irr(\Gbar)$-graded and for given
$\lambda\in \Irr(\Gbar)$ and $n\gg 0$ the isotypic component
$\cJ_n[\lambda]$ is zero by Lemma~\ref{ref:stability:lem}. Hence $\cJ_m=\cJ^{m+1}$ for every $m \in \NN$.
We define
\[
    Z=\Spec_{Z_0}(\cA)
\]
and we denote by $\pi:Z\to Z_0$ the structural morphism. The scheme $Z$
inherits a $\Gbar$-action from $\cA$. For
every $n\in \NN$ the zero-set of $\cJ^{n+1}\subseteq \cA$ is a $\Gbar$-scheme
isomorphic to $Z_n$. Hence $\Zformal$ is isomorphic to $\Zhat$.
Thus $Z$ is an algebraization of $\Zformal$. Since $\cA=\lim \cA_n$, we
have $Z = \colim Z_n$ in the category of locally linear $\Gbar$-schemes.

It remains to prove uniqueness of algebraization. Let $Z' = \Spec_{Z_0} \cA'$
be an algebraization of $\Zformal = \{Z_n\}$. Then $Z_n \into Z'$, so by the
universal property of colimit, we obtain a $\Gbar$-morphism $Z\to Z'$,
corresponding to $\cA' \to \cA$. It induces epimorphisms $\cA' \onto \cA_n$
for all $n$. For each $\lambda\in \Irr(\Gbar)$, the composition
\[
    \cA'[\lambda]\to \cA[\lambda]  \simeq \cA_{n_\lambda}[\lambda]
\]
is an epimorphism, hence $\cA'\to \cA$ is an epimorphism. The kernel of
$\cA'\to \cA$ is equal to
\[
    \bigcap_n \ker(\cA'\to \cA_n) = \bigcap_n \ker(\cA' \to \cA_0)^n.
\]
To prove that this kernel is zero, we may enlarge the field to an
algebraically closed field, so the result follows from
Corollary~\ref{ref:KempfTorus:cor}.

Assume that each scheme $Z_n$ is locally Noetherian over $\kk$. Then $\cI_n$
is a coherent $\cA_n$-module, thus $\cI_n^i/\cI^{i+1}$ is a coherent
$\cA_0$-module for all $i$. The series
\[
    0 = \cI_n^{n+1} \subset \cI^n \subset  \ldots \subset \cI \subset \cA_n
\]
has coherent subquotients, hence $\cA_n$ is a coherent $\cO_{Z_n}$-algebra.
Thus $\cA[\lambda]$ is a coherent $\cO_{Z_0}$-module for every
$\lambda\in \Irr(\Gbar)$. The claim that $\pi$ is of finite type is local on
$Z^{\Group}$, hence we may
assume that $Z^{\Group}$ is quasi-compact.
The sheaf $\cJ/\cJ^2\subseteq \cA_1$ is coherent so there exists a finite set
$\lambda_1, \ldots, \lambda_r\in \Irr(\Gbar)\setminus \{\lambda_0\}$ such that the morphism
\[
    \bigoplus_{i=1}^r\cA[\lambda_i]\to \cJ/\cJ^2
\]
induced by $\cA\twoheadrightarrow \cA_2$ is surjective. Let $\cB \subset \cA$
be the quasi-coherent $\cO_{Z_0}$-subalgebra generated by the coherent
subsheaf $\cM := \bigoplus_{i=1}^r\cA[\lambda_i]\subseteq \cA$.
Let $\kkbar$ be an algebraic closure of $\kk$ and let $\cA' = \cA \otimes
\kkbar$. Fix a Kempf's torus over
$\kkbar$ and the associated grading $\cA' = \bigoplus_{i\geq 0}
\cA'[i]$ as in
Corollary~\ref{ref:KempfTorus:cor}.
Then $\cJ = \bigoplus_{i\geq 1} \cA'[i]$ is a graded ideal and $\cJ/\cJ^2$ is
generated by the graded (coherent) subsheaf $\cM' = \bigoplus_{i=1}^r\cA'[\lambda_i]$. By
graded Nakayama's lemma, the ideal $\cJ$ itself is generated by (the elements
of) $\cM'$. Then by induction on the degree, $\cA'$ is generated by $\cM'$ as
an algebra. In other words, $\cA' = \cB\otimes \kkbar$. Thus also $\cA = \cB$ and so $\cA$ is of
finite type over $\cO_{Z_0}$.
\end{proof}

\newcommand{\varphihat}{\widehat{\varphi}}%
With the proof of Theorem~\ref{ref:algebraizationOfFormalSchemes:thm} in hand,
we can easily algebraize also equivariant mappings between formal schemes.

\begin{proposition}[Algebraization of morphisms of formal
    $\Gbar$-schemes]\label{ref:algebraizationOfMaps:prop}
    Let $\Wformal = \{W_n\}$ and $\Zformal = \{Z_n\}$ be formal $\Gbar$-schemes. Let $W$ and $Z$ be
    algebraizations of $\Wformal$ and $\Zformal$ respectively (see
    Theorem~\ref{ref:algebraizationOfFormalSchemes:thm}). Then every
    $\Gbar$-morphism $\varphihat\colon\Wformal\to \Zformal$  is the formalization of a unique
    $\Gbar$-equivariant morphism $\varphi\colon W\to Z$.
\end{proposition}
\begin{proof}
    The map $\varphihat$ induces maps $W_n \to Z_n \into Z$. By
    Theorem~\ref{ref:algebraizationOfFormalSchemes:thm}, the scheme $W$ is a
    colimit of $W_n$ in the category of locally linear $\Gbar$-schemes. By the universal
    property of the colimit, we obtain a unique $\Gbar$-equivariant morphism $W\to Z$.
\end{proof}

        \begin{example}\label{ex:Gbarnalgebraization}
            Clearly $\{\Gbar_n\}$ is a formal $\Gbar$-scheme. Applying
            Theorem~\ref{ref:algebraizationOfFormalSchemes:thm} we obtain  an
            algebraization which is isomorphic to $\Gbar$. This shows
            that the reductive monoid $\Gbar$ is uniquely determined by
            $\{\Gbar_n\}$ and $\Group$.
            More generally, the algebraization of the formal $\Gbar$-scheme $\{\Gbar_n
            \times S\}$ is $\Gbar \times S$ (here the $\Gbar$-action on $S$ is
            trivial).
        \end{example}

    We will now present the proof of representability of the formal
        \BBname{} functor $\Hfunctor{\varX}$. It is natural to do it for
        algebraic space $\varX$, hence in \'etale topology. The reader interested
        only in the case when $\varX$ is a scheme, can avoid \'etaleness by
        replacing each occurrence of ``\'etale morphism'' with ``open
        immersion''. The following lemma shows that the \'etale site has
        plenty of equivariant objects: the $\Group$-action on each
        infinitesimal neighbourhood of $\varX^{\Group} \subset \varX$ lifts to
        any of its \'etale neighbourhoods.
        \begin{lemma}\label{ref:liftingOfEtaleActions:lem}
            Let $Y$ be a scheme (resp. algebraic space) with a $\Group$-action
            $\sigma\colon\Group \times Y \to Y$ and assume that the ideal of
            $Y^{\Group} \subset Y$ is locally nilpotent.
            Let $\pi\colon U\to Y$ be an \'etale morphism. Then there is a unique
            $\Group$-action on $U$ such that $\pi$ is equivariant.
        \end{lemma}
        Note that without the assumption that $Y^{\Group} \subset Y$ is given
        by a locally nilpotent ideal, the claim of
        Lemma~\ref{ref:liftingOfEtaleActions:lem} is false, for example take
        an \'etale double  cover $E' = E\ni e \mapsto 2e\in E$ of an elliptic
        curve in characteristic zero.
        Then the action of $E$ on itself does not lift to $E'$.

        We also note that if one restricts to schemes (ignoring algebraic
        spaces), then the only case used below is when $\pi$ an open embedding.
        In this special case,
        Lemma~\ref{ref:liftingOfEtaleActions:lem} boils down to $\sigma(\Group
        \times U) \subset U$, which is obvious as $\sigma(g, u) = u$. The
        proof for \'etale $\pi$ follows the same idea, but in the language of
        the \'etale topos instead of Zariski topology.

        \begin{proof}[Proof of Lemma~\ref{ref:liftingOfEtaleActions:lem}]
            Let $Y' = Y^{\Group}$. Since $I(Y' \subset Y)$ is
            locally nilpotent, the morphism $Y'\to Y$ is a universal
            homeomorphism (i.e. for every $T$ the morphism $Y'\times T \to Y
            \times T$ is a homeomorphism).
            Let $U' = Y' \times_{Y} U$ and $\pi'\colon U'\to Y'$ and $i\colon U'\to U$ be the
            projections. Then $\pi'$ is \'etale and $i$ is a universal
            homeomorphism. Let $\sigma_{U'}\colon \Group \times U'
            \to U'$ be the second projection. Since the $\Group$-action on
            $Y'$ is trivial, we have
            \[
                \begin{tikzcd}
                    \Group \times U' \arrow[r, "\sigma_{U'}"]\arrow[d,
                    "{(\id_{\Group}, \pi')}"] & U'\arrow[d, "\pi'"]\\
                    \Group \times Y' \arrow[r, "\sigma_{|Y'}"] & Y'
                \end{tikzcd}
            \]
            Hence $\sigma_{U'}\in \Mor_{Y'}(\Group \times U', U')$.  By descent of morphisms with \'etale
            target ({\cite[Expos{\'e} IX, Proposition 3.2]{SGAI2}},
            {\cite[Tag~04DY]{stacks_project}}), there exists a unique map
            $\sigma_U\in \Mor_{Y}(\Group \times U, U)$ which pulls back to
            $\sigma_{U'}$ under $U' \to U$. Saying that $\sigma_U$ is a map of
            schemes over $Y$ amounts to the commutativity of the following
            diagram
            \[
                \begin{tikzcd}
                    \Group \times U \arrow[r, "\sigma_{U}"]\arrow[d,
                    "{(\id_{\Group}, \pi)}"] & U\arrow[d, "\pi"]\\
                    \Group \times Y \arrow[r, "\sigma_Y"] & Y
                \end{tikzcd}
            \]
            The morphism $\sigma_U$ is the required group action. Indeed, let
            $\mu\colon \Group \times \Group \to \Group$ be the multiplication,
            then $\sigma_U \circ (\mu \times \id_U), \sigma_U \circ
            (\id_{\Group} \times \sigma_U)\in \Mor(\Group \times \Group \times
            U, U)$ both pull back via $U'\to U$ to $\sigma_{U'} \circ (\mu \times \id_{U'})
            = \sigma_{U'} \circ (\id_{\Group} \times \sigma_{U'})$, hence are
            equal by uniqueness of the descent of morphisms with \'etale
            target. To prove that $\sigma_U$ is unique, fix any lift of
            $\Group$-action to $U$. The group $\Group$ is connected and acts
            on the fibers of $U'\to Y'$, hence
            it acts trivially on $U'$. Then the uniqueness of $\Group$-action
            follows from the uniqueness of the descent argument above.
        \end{proof}
        Now we prove that the restriction $\Dfunctor{X}\to\Hfunctor{X}$ is an
        isomorphism, i.e., formal actions algebraize. For this we need to
        consider some geometry of $\Dfunctor{X}$, $\Hfunctor{X}$. Recall that
        $\Maps(-, \varX)$ is an fpqc sheaf. Hence also $\Dfunctor{X}$,
        $\Hfunctor{X}$ are fpqc sheaves. This will be used to prove
        surjectivity of the restriction: we will first prove that the
        restriction is surjective on an fpqc cover and then conclude that it is
        surjective. In contrast, injectivity is straightforward, as shown
        in Lemma~\ref{ref:injectivityOfFormal:lem}.
        \begin{lemma}\label{ref:injectivityOfFormal:lem}
            For every $\Group$-algebraic space $X$ and every $\kk$-scheme $S$ the restriction
            $\Dfunctor{X}(S) \to \Hfunctor{X}(S)$ is injective. Hence,
            the morphism $\Dfunctor{X}\to \Hfunctor{X}$ of fpqc sheaves is injective.
        \end{lemma}
        \begin{proof}
            Taking a Zariski-open cover of $S$ by affine schemes we reduce to the case of affine $S$.
            Consider two $\Group$-equivariant morphisms $\varphi,\eta:\Gbar\times S\to X$ such that the corresponding formally $\Group$-equivariant families $\{\varphi_n\}_{n\geq 1}$ and $\{\eta_n\}_{n\geq 1}$ are equal. Let $h:K\to \Gbar\times S$ be the equalizer of the pair $\varphi,\,\eta$. Then $K$ is a $\Group$-scheme and $h$ is a $\Group$-equivariant locally closed immersion. Since $\varphi_n=\eta_n$ for every $n\geq 1$, we derive that for every $n\geq 1$ there is a factorization
            \begin{center}
                \begin{tikzpicture}
                    [description/.style={fill=white,inner sep=2pt}]
                    \matrix (m) [matrix of math nodes, row sep=3em, column sep=2em,text height=1.5ex, text depth=0.25ex] 
                    { K &  &    \Gbar\times S           \\
                    \Gbar_n\times S&   &  \\} ;
                    \path[->,font=\scriptsize]
                    (m-1-1) edge node[above] {$h $} (m-1-3);
                    \path[right hook->,font=\scriptsize]
                    (m-2-1) edge node[right] {$ $} (m-1-1)  
                    (m-2-1) edge node[left] {$  $} (m-1-3);
                \end{tikzpicture}
            \end{center}
            Let $o = 0_{\Gbar}$.
            The ideal $I(\Gbar_n \times S) \subset \cO_{\Gbar \times S}$ is
            equal to $I(o)^n \otimes
            \cO_S$. Since $\Gbar$ is of finite type, the intersection of these ideals is zero in
            $\cO_{\Gbar \times S}$. Thus $K$ contains the subscheme $\Spec
            \cO_{\Gbar, o} \times S$. But $\Spec \cO_{\Gbar, o} \times S\to
            \Gbar \times S$ is dominant, so there are no functions vanishing
            on $K$, in other words $K$ is an open subscheme. Then its
            complement is closed, $\Group$-stable and does not intersect $o \times S$, hence is
            empty by Lemma~\ref{ref:smallestorbit:lem} applied to the fibers
            of $\Gbar \times S\to S$. Thus $K = \Gbar \times S$ and so
            $\varphi=\eta$.
        \end{proof}
        Before we prove the representability in full generality we need an
        important special case: affine schemes.
        We note that the functors $\Dfunctor{X}$, $\Hfunctor{X}$ are sheaves
        in any sub-canonical topology, in particular in the fpqc topology.

        \begin{lemma}\label{ref:agreementForAffine:lem:newversion}
            Let $W$ be an affine scheme over $\kk$ equipped with a
            $\Group$-action. Then the restriction $\Dfunctor{W}\to \Hfunctor{W}$ is an
            isomorphism. If $W$ is a $\Gbar$-scheme, then both functors are represented by $W$.
        \end{lemma}
        \begin{proof}
        \def\Wplus{W^+}%
        Let $A = H^0(W, \cO_W)$. By
        Proposition~\ref{ref:representabilityForAffine:prop} we know that
        $\Dfunctor{W}$ is represented by a closed subscheme $\Wplus =
        \Spec(A/I)$ of $W$, where $I$ is the ideal generated by all outsider
        representations in $A$. By the same
        proposition (or by Proposition~\ref{ref:isomorphism:prop}) also
        $\Dfunctor{\Wplus}$ is represented by $\Wplus$. In particular,
        the natural map $\Dfunctor{\Wplus}\to \Dfunctor{W}$ is an isomorphism.

        For any $n\geq 0$ and any scheme $S$, a $\Group$-equivariant map
        $\varphi_n\colon\Gbar_n \times S\to W$ corresponds to a map of
        representations $\varphi^{\#}\colon A\to
        H^0(\Gbar_{n}, \cO_{\Gbar_n})\otimes H^0(S, \cO_S)$. The
        right hand side is a $\Gbar$-representation, hence $\ker
        \varphi^{\#}$ contains all outsider representations of $A$, see
        Lemma~\ref{ref:mainRepr:lem}. The map
        $\varphi^{\#}$ is a homomorphism, hence $\ker \varphi^{\#}$ contains
        $I$. Thus, $\varphi_n$ factors through
        $\Wplus\into W$. Therefore, the natural map $\Hfunctor{\Wplus}\to
        \Hfunctor{W}$ is an isomorphism.

        Consider the Commutative Square~\eqref{eq:restrictionsSquare}, whose
        vertical arrows come functorially from $W^+\to W$ and horizontal
        arrows come from the restriction $\Dfunctor{}\to \Hfunctor{}$. By
        the above discussion, the vertical maps are isomorphisms.
        \begin{equation}\label{eq:restrictionsSquare}
                \begin{tikzpicture}
                    [description/.style={fill=white,inner sep=2pt}]
                    \matrix (m) [matrix of math nodes, row sep=3em, column sep=3em,text height=1.5ex, text depth=0.25ex] 
                    { \Dfunctor{W^+} &     \Hfunctor{W^+}           \\
                      \Dfunctor{W} &   \Hfunctor{W}  \\} ;
                    \path[->,font=\scriptsize]
                    (m-1-1) edge node[above] {$ $} (m-1-2)
                    (m-1-1) edge node[left] {$\cong $} (m-2-1)  
                    (m-1-2) edge node[right]  {$\cong $} (m-2-2)
                    (m-2-1) edge node[left]  {$ $} (m-2-2);
                \end{tikzpicture}
            \end{equation}
        To conclude that $\Dfunctor{W}\to \Hfunctor{W}$ is an
        isomorphism, it suffices to prove that $\Dfunctor{W^+}\to
        \Hfunctor{W^+}$ is an isomorphism. Hence, we may and do restrict attention to
        the case when $W$ is a $\Gbar$-scheme.
            Fix a scheme $S$.
            By Lemma~\ref{ref:injectivityOfFormal:lem}, the map
            $\Dfunctor{W}(S)\to \Hfunctor{W}(S)$ is injective. It remains to prove
            that $\Dfunctor{W}(S)\to \Hfunctor{W}(S)$ is surjective. Fix an
            element of $[\gamma]\in \Hfunctor{W}(S)$. It corresponds to a coherent
            sequence of families $\gamma_n\colon \Gbar_n \times S\to W$, hence to a morphism
            $\widehat{\gamma}\colon \{\Gbar_n \times S\} \to \What$ of formal
            $\Gbar$-schemes.  As $W$ is affine, it is locally linear, hence it
            is an algebraization of
            $\What$ (see~Definition~\ref{ref:algebraization:def}). By
            Proposition~\ref{ref:algebraizationOfMaps:prop} (see
            Example~\ref{ex:Gbarnalgebraization}) the morphism
            $\gamma$ comes from a unique $\Gbar$-equivariant morphism
            $\gamma\colon\Gbar \times S\to W$. Thus $[\gamma]$ lies in the
            image of $\Dfunctor{W}(S)\to \Hfunctor{W}(S)$.
        \end{proof}
        \begin{corollary}\label{ref:agreementForLocallyLinear:cor}
            Let $Z$ be a locally linear $\Gbar$-scheme. Then the restriction
            map $\Dfunctor{Z}\to \Hfunctor{Z}$ is an isomorphism.
        \end{corollary}
        \begin{proof}
            Fix an affine open $\Gbar$-stable cover $\{V_i\}_i$ of $Z$. By the
            discussion after Definition~\ref{ref:locallylinear:def}, we have
            $\iinfty{Z}^{-1}(\iinfty{Z}(V_i)) = V_i$ for all $i$.
            By Proposition~\ref{ref:openimmersionscartesian:prop} and its
            direct analogue for $\Hfunctor{Z}$, we have a cartesian diagram
            \[
                \begin{tikzcd}
                    \Dfunctor{V_i} \arrow[r]\arrow[d, hook] & \Hfunctor{V_i} \arrow[r,
                    ]\arrow[d, hook] & V_i^{\Group}\arrow[d, hook]\\
                    \Dfunctor{Z} \arrow[r] & \Hfunctor{Z} \arrow[r,
                    ] & Z^{\Group}
                \end{tikzcd}
            \]
            Moreover, by Point~\ref{it:opentwo} of
            Proposition~\ref{ref:openimmersions:prop},
            the functor
            $\Dfunctor{Z}$ is covered by $\{\Dfunctor{V_i}\}_i$. By a direct
            analogue of this proposition, also $\Hfunctor{Z}$ is covered by
            $\{\Hfunctor{V_i}\}_i$. The natural restriction map $\coprod
            \Dfunctor{V_i}\to \coprod \Hfunctor{V_i}$ is an isomorphism by
            Lemma~\ref{ref:agreementForAffine:lem:newversion}, hence also
            $\Dfunctor{Z}\to \Hfunctor{Z}$ is an isomorphism by descent.
        \end{proof}

        \begin{theorem}\label{ref:representabilityformal:thm}
            Let $X$ be an algebraic space over $\kk$ together with a $\Group$-action.
            Then the functor $\Hfunctor{X}$ is representable by an
            algebraic space and the morphism $\Hfunctor{X}\to \Ffunctor{X}$ is affine.
            Moreover,
            \begin{enumerate}
                \item if $\varX$ is a scheme, then $\Hfunctor{X}$ is represented by a
                    scheme,
                \item if $\varX$ is locally Noetherian, then $\pi$ is of finite type.
                    Hence $\Hfunctor{X}$ is also locally Noetherian. If moreover
                    $\varX^{\Group}$ is of finite type, then $\Hfunctor{X}$ is of finite
                    type.
            \end{enumerate}
        \end{theorem}
\begin{proof}
Let $p\colon U\to X$ be an {\'e}tale surjective morphism from a scheme (if
$\varX$ is locally Noetherian over $\kk$ then we choose $U$ locally
Noetherian). Restricting $U$ to infinitesimal neighbourhoods of
$\varX^{\Group}$ we obtain the following commutative diagram
\begin{center}
\begin{tikzpicture}
[description/.style={fill=white,inner sep=2pt}]
\matrix (m) [matrix of math nodes, row sep=3em, column sep=2em,text height=1.5ex, text depth=0.25ex] 
{U_0 & & ... & & U_n & & U_{n+1} & & ... & & U       \\
 X_0 & & ... & & X_n & & X_{n+1} & & ... & & X\\} ;
\path[right hook->,font=\scriptsize] 
(m-1-1) edge node[above] {$ $} (m-1-3)
(m-1-3) edge node[above] {$ $} (m-1-5)
(m-1-5) edge node[above] {$ $} (m-1-7)
(m-1-7) edge node[above] {$ $} (m-1-9)
(m-1-9) edge node[above] {$ $} (m-1-11)
(m-2-1) edge node[above] {$ $} (m-2-3)
(m-2-3) edge node[above] {$ $} (m-2-5)
(m-2-5) edge node[above] {$ $} (m-2-7)
(m-2-7) edge node[above] {$ $} (m-2-9)
(m-2-9) edge node[above] {$ $} (m-2-11);
\path[->,font=\scriptsize]
(m-1-1) edge node[left] {$ p_0$} (m-2-1)
(m-1-5) edge node[left] {$ p_n$} (m-2-5)
(m-1-7) edge node[left] {$ p_{n+1}$} (m-2-7)
(m-1-11) edge node[right] {$ p$} (m-2-11);
\end{tikzpicture}
\end{center}
Using  Lemma~\ref{ref:liftingOfEtaleActions:lem} we obtain a unique $\Group$-action on
each $U_n$, so that the maps in Diagram~\eqref{eq:diagramFormalABBEtale} below
become equivariant
\begin{equation}\label{eq:diagramFormalABBEtale}
\begin{tikzpicture}
[description/.style={fill=white,inner sep=2pt}]
\matrix (m) [matrix of math nodes, row sep=3em, column sep=2em,text height=1.5ex, text depth=0.25ex] 
{U_0 & & ... & & U_n & & U_{n+1} & & ... & &        \\
 X_0 & & ... & & X_n & & X_{n+1} & & ... & & X\\} ;
\path[right hook->,font=\scriptsize] 
(m-1-1) edge node[above] {$ $} (m-1-3)
(m-1-3) edge node[above] {$ $} (m-1-5)
(m-1-5) edge node[above] {$ $} (m-1-7)
(m-1-7) edge node[above] {$ $} (m-1-9)
(m-2-1) edge node[above] {$ $} (m-2-3)
(m-2-3) edge node[above] {$ $} (m-2-5)
(m-2-5) edge node[above] {$ $} (m-2-7)
(m-2-7) edge node[above] {$ $} (m-2-9)
(m-2-9) edge node[above] {$ $} (m-2-11);
\path[->,font=\scriptsize]
(m-1-1) edge node[left] {$ p_0$} (m-2-1)
(m-1-5) edge node[left] {$ p_n$} (m-2-5)
(m-1-7) edge node[left] {$ p_{n+1}$} (m-2-7);
\end{tikzpicture}
\end{equation}
In particular, $\widehat{U}=\{U_n\}_{n\in \NN}$ is a formal
$\Group$-scheme. For every $n$, every Zariski-open subset of $U_n$ is
$\Group$-stable, hence $\Dfunctor{U_n}$ is represented by a closed subscheme
$Z_n \subset U_n$, see Proposition~\ref{ref:openimmersions:prop}.
By Proposition~\ref{ref:closedImmersions:prop}, for each $n$, $m$ the
Diagram~\eqref{eq:diagramABBFormalClosed} below is cartesian.
\begin{equation}\label{eq:diagramABBFormalClosed}
\begin{tikzpicture}
[description/.style={fill=white,inner sep=2pt}]
\matrix (m) [matrix of math nodes, row sep=3em, column sep=2em,text height=1.5ex, text depth=0.25ex] 
{Z_m&  &    Z_{n}         \\
U_{m}&   & U_{n}       \\} ;
\path[right hook->,font=\scriptsize] 
(m-1-3) edge node[right] {$ $} (m-2-3)
(m-1-1) edge node[left] {$ $} (m-2-1)
(m-1-1) edge node[above] {$ $} (m-1-3)
(m-2-1) edge node[below] {$ $} (m-2-3);
\end{tikzpicture}
\end{equation}
Let $\cI_n$ be a quasi-coherent ideal defining $U_0$ in $U_n$ and $\cJ_n$ be
an ideal defining $Z_0=U_0$ in $Z_n$. Since
Diagram~\eqref{eq:diagramABBFormalClosed} is cartesian and the kernel of the
surjective morphism $\cO_{U_n}\twoheadrightarrow \cO_{U_m}$ is $\cI_n^{m+1}$,
we derive that the kernel of the surjective morphism $\cO_{Z_n}\to \cO_{Z_m}$ is
$\cJ_n^{m+1}$. Hence $\Zformal=\{Z_n\}_{n\in \NN}$ is a formal $\Gbar$-scheme.
Moreover, $\Zformal$ is locally Noetherian if $\varX$ is locally Noetherian. By
Theorem \ref{ref:algebraizationOfFormalSchemes:thm}, we deduce that there exists a $\Gbar$-scheme
$Z$ such that
\begin{enumerate}
    \item $\Zformal$ and $\Zhat$ are isomorphic formal $\Gbar$-schemes.
    \item The canonical morphism $\iinfty{Z}:Z\to Z^{\Group}$ is affine.
    \item If $\varX$ is locally Noetherian then $\iinfty{Z}$ is of finite type.
\end{enumerate}
Arguing as in Theorem~\ref{ref:etale:thm} we have a cartesian diagram of {\'e}tale sheaves
\[
    \begin{tikzpicture}
        [description/.style={fill=white,inner sep=2pt}]
        \matrix (m) [matrix of math nodes, row sep=3em, column sep=2em,text height=1.5ex, text depth=0.25ex] 
        {\Hfunctor{Z}&  &      \Hfunctor{X}       \\
        Z^{\Group}&   & \Ffunctor{X}       \\} ;
        \path[->,font=\scriptsize]
        (m-1-3) edge node[right] {$ $} (m-2-3)
        (m-1-1) edge node[left] {$\iinfty{Z}$} (m-2-1)
        (m-1-1) edge node[above] {$ $} (m-1-3)
        (m-2-1) edge node[below] {$ $} (m-2-3);
    \end{tikzpicture}
\]
Since $\iinfty{Z}$ is affine, $Z$ is a locally linear $\Gbar$-scheme. The morphism $\Hfunctor{Z}\to Z$ is an isomorphism by
Corollary~\ref{ref:agreementForLocallyLinear:cor}.
By~\cite[Exercise~5.G]{Olsson}, $\Hfunctor{\varX}$ is an algebraic space. If
$\varX$ is a scheme, we can take $U = \varX$ and deduce that
$\Hfunctor{\varX}$ is represented by a scheme $Z$. Finally, if $\varX$ is
(locally) Noetherian, then $\iinfty{Z}$ is of finite type hence $\Hfunctor{X} \to
\varX^{\Group}$ is of finite type. This completes the proof.
\end{proof}

        By Lemma~\ref{ref:agreementForAffine:lem:newversion}, the functors
        $\Dfunctor{\varX}$, $\Hfunctor{\varX}$ agree for affine $\varX$.
        To prove that they agree everywhere,  we apply
        Theorem~\ref{ref:etalestableneighbourhoods:thm} to find a surjective
        equivariant morphism $W\to \varX$ from $W$ isomorphic to a disjoint union
        of affine schemes.

        \begin{theorem}[Existence of \BBname{} for algebraic
            spaces]\label{ref:Representability:thm}
            Let $\varX$ be a quasi-separated algebraic space locally of finite
            type over $\kk$. Then, the restriction $\Dfunctor{X} \to
            \Hfunctor{X}$ is an isomorphism and both functors are represented by
            an algebraic space $\Xplus$
            locally of finite type over $\kk$. The map $\iinftyX\colon \Xplus\to
            \varX^{\Group}$ is affine of finite type.
        \end{theorem}
        We remark that both Theorem~\ref{ref:introRepresentability:thm} and
        Theorem~\ref{ref:algebraizationABB} follow from
        Theorem~\ref{ref:Representability:thm} in the case $\varX$ is a
        scheme.

        \begin{proof}[Proof of~Theorem~\ref{ref:Representability:thm}]
            By Lemma~\ref{ref:injectivityOfFormal:lem} the transformation
            $\Dfunctor{X}\to \Hfunctor{X}$ is injective. Let us prove that
            $\Dfunctor{X} \to \Hfunctor{X}$ is surjective.
            Let $\kkbar$ be an algebraic closure of $\kk$. According to
            Theorem~\ref{ref:etalestableneighbourhoods:thm} there exists
            a $\Group_{\kkbar}$-equivariant
            {\'e}tale morphism $f: W \to X_{\kkbar}$ such that $W$ is a
            disjoint union of affine
            $\Group_{\kkbar}$-schemes and the image of $f$ contains the set of fixed
            points of $\Group_{\kkbar}$ on $X_{\kkbar}$.
            Consider Diagram~\eqref{eq:Trautmancover}, which is a commutative
            square of fpqc sheaves.
            According to Theorem~\ref{ref:etale:thm}
            and Lemma~\ref{ref:agreementForAffine:lem:newversion}, the top
            horizontal arrow and right vertical arrow are surjections.
            \begin{equation}\label{eq:Trautmancover}
                \begin{tikzpicture}
                    [description/.style={fill=white,inner sep=2pt}]
                    \matrix (m) [matrix of math nodes, row sep=3em, column sep=2em,text height=1.5ex, text depth=0.25ex] 
                    {\Dfunctor{W} &  &    \Hfunctor{W}           \\
                    \Dfunctor{X_{\kkbar}}&   &   \Hfunctor{X_{\kkbar}} \\} ;
                    \path[->,font=\scriptsize]
                    (m-2-1) edge node[below] {$ $} (m-2-3)
                    (m-1-1) edge node[left] {$ \Dfunctor{f} $} (m-2-1);
                    \path[->>,font=\scriptsize]
                    (m-1-1) edge node[above] {$ $} (m-1-3)
                    (m-1-3) edge node[right] {$ \Hfunctor{f}$} (m-2-3);
                \end{tikzpicture}
            \end{equation}
            Thus also $\Dfunctor{X_{\kkbar}}\to \Hfunctor{X_{\kkbar}}$ is a surjection of fpqc sheaves. This implies that
            $\Dfunctor{X}\to \Hfunctor{X}$
            is a surjective morphism of fpqc sheaves after restriction to schemes over
            $\kkbar$. For every $\kk$-scheme $Z$ the projection $Z_{\kkbar}\to Z$ is an
            fpqc cover. Hence the morphism
            $\Dfunctor{X}\to \Hfunctor{X}$
            is a surjection of fpqc sheaves. As it is also injective, it is an
            isomorphism. By Theorem~\ref{ref:representabilityformal:thm}, the
            functor $\Hfunctor{X}$ is representable and affine of finite type over
            $\varX^{\Group}$.
        \end{proof}

        \begin{proof}[Proof of Proposition~\ref{ref:comparison:intro:prop}]
            \def\mm{\mathfrak{m}}%
            \def\Uplus{U^{+}}%
            First, we prove that $\ioneX^{\#}\colon \hat{\cO}_{\varX, x} \to
            \hat{\cO}_{\Xplus, x'}$ is an isomorphism. Let $\mm = \mm_x$. For all $i > 0$,
            the $\Group$-representation $\hat{\cO}_{\varX, x}/\mm^i$ has a
            composition series with subquotients $\mm^j/\mm^{j+1}$, for $j=0,
            \ldots, i-1$. For all such $j$ we have surjections
            $\Sym^j(\mm/\mm^2) \onto \mm^j/\mm^{j+1}$. By assumption, there
            are  no outsider representations in $\mm/\mm^2$, hence there are
            no outsider representations in $\Sym^j(\mm/\mm^2)$ and in
            $\mm^j/\mm^{j+1}$, thus there are no outsider representations in
            $\hat{\cO}_{\varX, x}/\mm^i$. Then $\Spec(\hat{\cO}_{\varX,
            x}/\mm^i)$ is a $\Gbar$-scheme for all $i$, hence $\hat{\cO}_{\varX,
            x}/\mm^i \to \hat{\cO}_{\Xplus, x'}/\mm_{x'}^i$ is an isomorphism
            for all $i$, see Proposition~\ref{ref:isomorphism:prop}. Then $\ioneX^{\#}\colon \hat{\cO}_{\varX, x} \to
            \hat{\cO}_{\Xplus, x'}$ is an isomorphism, so $\ioneX$ is
            \'etale at $x'$ by~\cite[Prop~17.11.2]{EGA4-4}. The \'etale locus
            of $\ioneX$ is an open and $\Group$-stable subscheme of $\Xplus$. Denote it by $U$. The map
            $(\ioneX)_{|U}$ is \'etale and a
            monomorphism by Lemma~\ref{ref:monomorphism:lem} hence it is an
            open immersion~\cite[Tag~025G]{stacks_project}.

            Let $V = \isectionX^{-1}(U) \subset \varX^{\Group}$ be the open
            subset of
            $\Group$-fixed points in $U$. We have $x'\in
            \isectionX(\varX^{\Group}) \cap U$, hence $V$ is
            non-empty.  The subscheme $U\cap \iinftyX^{-1}(V)$ is open and
            contains $\isectionX(V)$. Applying
            Proposition~\ref{ref:neighbourhoodsofzero:prop} to fibers of
            $\iinftyX$, we see that $U \supset \iinftyX^{-1}(V)$.
            Choose any affine open $V' \subset V$. Then $U' = \iinftyX^{-1}(V')$ is
            affine, open, $\Gbar$-stable and contained in $U$.
            Then $(\ioneX)_{|U'}\colon U' \to \varX$ is an open immersion, whose
            image is an affine, $\Gbar$-stable neighbourhood of $x$.

            Conversely, suppose that $x$ has an affine $\Gbar$-stable
            neighbourhood $Z$. Then $Z$ is a $\Gbar$-scheme, hence
            $Z^{+} = Z$ by Proposition~\ref{ref:isomorphism:prop}. By Proposition~\ref{ref:affineRepresentability:prop}
            the ideal $I(Z^+ \subset Z)$ is generated by outsider
            representations in $H^0(Z, \cO_Z)$. Putting this together, we see
            that there are no outsider representations in $H^0(Z, \cO_Z)$,
            hence also no outsider representations in its sub-quotient
            $T_{\varX, x}^{\vee} = \mm_x/\mm_x^2$.
        \end{proof}

\section{Smoothness and relative \BBname{} decompositions}\label{sec:smoothnessAndRelative}
    In this section we deal with the smooth case. Let $\Gbar$ and $\Group$ be
    as before (see the beginning of Section~\ref{sec:BBproperties}).

    \begin{theorem}\label{ref:relativesmoothness:thm}
        Let $\Gbar$ be a linearly reductive monoid with a zero and $\Group$ be its unit
        group.  Assume that $X$ is a $\Group$-scheme locally of finite type over $\kk$. Let $f\colon X\to Y$ be a $\Group$-equivariant and smooth
        morphism.  Then the morphism $\Dfunctor{f}\colon \Dfunctor{X}\to
        \Dfunctor{Y}$ is also smooth.
    \end{theorem}
    \begin{proof}
        \def\mm{\mathfrak{m}}%
        \def\nn{\mathfrak{n}}%
        Note that for every field $L \supset \kk$ we have $\Dfunctor{Z_{L}}
        =\Dfunctor{Z} \times_{\kk} \Spec(L)$.
        It is enough to prove that $\Dfunctor{f}$ is smooth near each point
        $x\in \Dfunctor{\varX}(K)$ for any field $K$. Choose $x$. After a base
        change we may assume that $\kk = \kkbar$ and that $x$ is a
        $\kk$-point. Let $x_0 = \iinfty{X}(x)$ be its limit point. By
        Theorem~\ref{ref:etalestableneighbourhoods:thm}, there exists an
        \'etale $\Group$-equivariant
        morphism $g\colon W\to \varX$ such that $W$ is affine and $x_0\in g(W)$. Since $g(W)$ is
        $\Group$-stable and $x_0\in g(W)$, we have also $x\in
        \Dfunctor{W}(\kk)$.
        The map $\Dfunctor{g}\colon \Dfunctor{W}\to \Dfunctor{\varX}$ is
        \'etale by Theorem~\ref{ref:etale:thm} and
        Theorem~\ref{ref:algebraizationABB}.
        Thus, we may replace $\varX$ by $W$ and hence further we assume that
        $\varX = \Spec(B)$ is affine.

        Now we prove that there exists an open $\Group$-stable neighbourhood $x\in U\subset
        \varX$ and a decomposition of $f_{|U}$ into \emph{equivariant maps}
        $f_{|U}\colon U \to Y \times \mathbb{A}^n \to Y$, where the first map
        is \'etale and the second is the first projection.

        Let $\mm \subset B$ be the maximal ideal corresponding to $\kk$-point
        $x_0$. Since $x_0$ is fixed, the ideal $\mm$ is $\Group$-stable.
        Let $y_0 = f(x_0)$. Then $y_0$ is also $\Group$-fixed. Let $\nn
        \subset B$ be the pullback of the ideal of $y_0\in Y$ and let
        \[
            V = \frac{\mm}{\mm^2 + \nn}
        \]
        be the relative cotangent space at $x_0$.
        Since $\nn
        \subset \mm$ is $\Group$-stable, the linear space $V$ has a natural
        $\Group$-action. Since $\Group$ is linearly reductive, the surjection $\mm \to
        V$ of $\kk$-linear spaces has a $\Group$-equivariant
        section
        \[
            j\colon V\to \mm \subset B.
        \]
        Let $\mathbb{A} := \Spec\Sym_{\kk}(V)$ be an affine space.
        Then $j$ induces a morphism $j'\colon \varX\to \mathbb{A}$. Since $j$ is
        equivariant, this morphism is also equivariant. Hence we obtain
        \[
            \begin{tikzcd}
                \varX \arrow[r, "j' \times f"]\arrow[d, "f"]  &\mathbb{A} \times
                Y\arrow[ld, "pr_2"]\\
                Y
            \end{tikzcd}
        \]
        By construction, the tangent map of $j' \times f$ is an isomorphism at
        $x_0$. Both $\varX$ and $\mathbb{A} \times Y$ are smooth over $Y$,
        hence $j' \times f$ is \'etale at $x_0$
        by~\cite[Prop~17.11.2]{EGA4-4}. Hence this map is
        \'etale on an open subset $U$ containing $x_0$. Since it is
        equivariant, this open subset may be chosen equivariant (i.e. by
        taking the union of all translates). Restricting to $U$ we obtain the
        desired decomposition. Note that $x_0\in U$ and $U$ is stable, hence
        $x\in \Dfunctor{U}$, see Proposition~\ref{ref:openimmersions:prop}.
        Applying $\Dfunctor{-}$ to the diagram, we have
        \[
            \begin{tikzcd}
                \Dfunctor{U} \arrow[r, "\Dfunctor{j' \times
                f}"]\arrow[d, "\Dfunctor{f}"]
                &\Dfunctor{\mathbb{A}} \times
                \Dfunctor{Y}\arrow[ld, "pr_2 = \Dfunctor{pr_2}"]\\
                \Dfunctor{Y}
            \end{tikzcd}
        \]
        Now, by Proposition~\ref{ref:representabilityForAffine:prop},
        $\Dfunctor{\mathbb{A}}$ is an affine subspace of $\mathbb{A}$, cut out by all
        outsider representations in $\mm/\mm^2$. Hence $\Dfunctor{\mathbb{A}}
        \times \Dfunctor{Y}\to \Dfunctor{Y}$ is smooth. Moreover,
        $\Dfunctor{j' \times f}$ is \'etale by
        Theorem~\ref{ref:etale:thm},
        since $j'\times f$ is \'etale and equivariant. Finally
        $\Dfunctor{f_{|U}}\colon \Dfunctor{U}\to \Dfunctor{Y}$ is smooth as
        composition of smooth morphisms.
        But $\Dfunctor{U} \subset \Dfunctor{X}$ is open
        (Proposition~\ref{ref:openimmersions:prop}) and contains $x$, so $\Dfunctor{X}\to
        \Dfunctor{Y}$ is smooth at a neighbourhood of $x$.
    \end{proof}

    To prove the classical \BBname{} decomposition, recall that we have
    morphisms of functors
    \[
        \begin{tikzcd}
            \Dfunctor{\varX} \arrow[r, "\ioneX"]\arrow[d, "\iinftyX"] & \varX\\
            \varX^{\Group}\arrow[u, "\isectionX", bend left]
        \end{tikzcd}
    \]
    where $\ioneX$, as defined in the introduction, is the restriction
    $\varphi\colon \Gbar
    \times S \to \varX$ to $\varphi_{|1 \times S}\colon S\to \varX$,
    $\iinftyX$ is the restriction $\varphi_{|0 \times S}\colon S\to
    \varX^{\Group}$ and $\isectionX$ is the section of $\iinftyX$ which takes
    a family $f\colon S\to \varX^{\Group}$ to the constant family $g\colon \Gbar \times S \to S \to
    \varX^{\Group}$. Note that $g$ is equivariant.

    \begin{lemma}\label{ref:conesWhichAreSmoothAreAffineSpaces:lem}
        Let $Z$ be a scheme and $\cB = \bigoplus_{i\geq 0} \cB_i$ with $\cB =
        \cO_{Z}$ be a sheaf of algebras, such that $\varphi\colon \Spec(\cB)\to Z$ is
        smooth. Then $\varphi\colon \Spec(\cB)\to Z$ is an affine fiber bundle.
    \end{lemma}
    \begin{proof}
        The claim is local on $Z$. Choose a point $z\in Z$ with residue field
        $\kappa(z)$ and a minimal
        system $f_1, \ldots ,f_r$ of homogeneous generators of the ideal $\cB_{>0}\otimes \kappa(z)$
        of the algebra $\cB\otimes \kappa(z)$.

        Shrink $Z$ to an affine neighbourhood
        $z\in \Spec(A) \subset Z$ so that those generators lift to
        generators $F_1, \ldots ,F_r$ of $\cB_{>0}$. On the one hand, we have
        a natural morphism $p\colon \cB_0[F_1, \ldots ,F_r]\to \cB$. Since $F_1,
        \ldots ,F_r$ are homogeneous and generate the ideal $\cB_{>0}$, the
        morphism $p$ is surjective by induction on the degrees.

        On the other hand, note that $f_1, \ldots ,f_r$ are linearly
        independent in the relative cotangent space of $\varphi$. As the
        fibers are smooth, we have $r\leq \dim \varphi^{-1}(z) = \dim
        \Spec(\cB\otimes \kappa(z))$. Since $\varphi$ is flat, we have
        \[
            \cB\otimes \kappa(z) = \frac{\kappa(z)[F_1, \ldots ,F_r]}{(\ker p)\otimes
            \kappa(z)}.
        \]
        The right hand side has dimension $r$ only if $(\ker p)\otimes
        \kappa(z) = 0$. Thus (after possibly shrinking $Z$ again), we have
        $\ker p =0$ and $\cB  \simeq \cB_0[F_1, \ldots ,F_r]$, so $\Spec(\cB)
         \simeq \mathbb{A}^r \times Z$.
    \end{proof}

    Since the result itself is classical, we formulate it classically, in
    terms of $\Xplus$ representing $\DX$ (this representing scheme exists by
    Theorem~\ref{ref:algebraizationABB}, which was proven in the previous
    section).
    \begin{corollary}\label{ref:affineBundle:cor}
        Suppose that $\varX$ is a smooth variety over $\kk$. Then the scheme
        $\Xplus$, representing $\DX$, is an affine fiber bundle over
        $\varX^{\Group}$. Moreover, both $\varX^{\Group}$ and $\Xplus$ are
        smooth.
    \end{corollary}
    \begin{proof}
         By Theorem~\ref{ref:relativesmoothness:thm} applied to $\varX\to
         \Spec(\kk)$, we see that $\Xplus$ is smooth over $\Spec(\kk)^{+} =
         \Spec(\kk)$. Classically, $\varX^{\Group}$ is smooth.
         Consider
         \[
             \iinftyX\colon \Xplus \to \varX^{\Group}.
         \]
         By Theorem~\ref{ref:algebraizationABB}, this map is affine
         of finite type, so $\Xplus  \simeq \Spec \cB$ for a finite type
         $\cO_{\varX^{\Group}}$-algebra.  We have a section $\isectionX$ of
         $\iinftyX$, so $\iinftyX$ is surjective on tangent spaces for every
         point of this section and hence $\iinftyX$ is smooth at every point
         of $\isectionX$. The non-smooth locus of $\iinftyX$ is a closed
         subscheme $Z \subset \Xplus$ invariant under the $\Gbar$-action. For
         every geometric point $z\in Z(\kkbar)$, we would have $0\cdot z\in
         \isectionX(\varX^{\Group})$ and $0\cdot z\in Z$, hence a contradiction. Thus $Z$
         is empty and so $\iinftyX$ is smooth.

         To prove that $\Xplus \to
         \varX^{\Group}$ is an affine fiber bundle, we pass to the
         algebraic closure of $\kk$.
         Then we fix a Kempf's torus
         $\mathbb{A}^1\to \Gbar$ and restrict the $\Gbar$-action on $\Xplus$
         to $\mathbb{A}^1 \times \Xplus \to \Xplus$. This action gives $\cB$
         an $\mathbb{N}$-grading: $\cB = \bigoplus_{i\geq 0} \cB_i$. The claim
         follows from Lemma~\ref{ref:conesWhichAreSmoothAreAffineSpaces:lem}.
    \end{proof}

    \begin{example}
        Let $\varX$ be as in Corollary~\ref{ref:affineBundle:cor}.
        It is well-known, though non-trivial, that $\Xplus\to \varX^{\Group}$ may fail to be a
        vector bundle, see~\cite{Iarrobino__fibre_bundle}.
    \end{example}

    \begin{remark}\label{ref:locallyLinearSchemes:rmk}
        Theorem~\ref{ref:relativesmoothness:thm} and
        Corollary~\ref{ref:affineBundle:cor} are may be proven without using
        the full power Theorem~\ref{ref:algebraizationABB}, in particular without
        mentions of \'etale covers etc. This is because $\varX$ is assumed to
        be normal, hence every its $\Group$-fixed point has an affine
        Zariski-open $\Group$-neighbourhood by~\cite[Theorem~4.1]{Trautman}.
    \end{remark}

    \begin{proposition}\label{ref:ioneLocallyClosed:prop}
        Let $\varX$ be a geometrically normal variety over $\kk$ and $\Gbar$
        be a monoid with zero.
        Let $\Xplus$ represent $\Dfunctor{\varX}$ and let $Z$ be one of its
        irreducible components, with reduced structure. Then the composition $Z \into \Xplus
        \to \varX$ is a locally closed immersion.
    \end{proposition}
    \begin{proof}
        \def\Uplus{U^{+}}%
        The claim is local on $\varX$. Pick a geometric point $x\in
        \varX(\kkbar)$ lying in $\ioneX(Z)$. We will prove the claim after
        restriction to a certain open neighbourhood of $x\in \varX$.

        Let $z\in Z(\kkbar)$ be such
        that $\ioneX(z) = x$ and let $z_0 = \iinftyX(z)\in
        \varX^{\Group}(\kkbar)$
        be its limit point. By Trautman's Sumihiro-type
        theorem~\cite[Theorem~4.1]{Trautman}, any point of $\varX_{\kkbar}$
        lying over $z_0$ has an open
        affine $\Group$-stable neighbourhood $U' \subset \varX_{\kkbar}$. Let
        $U$ be the image of $U'$ under the projection $\varX_{\kkbar}\to
        \varX$. Then $U$ is open, $\Group$-stable,
        affine~\cite[Tag~05YU]{stacks_project} and contains $z_0$.
        By Proposition~\ref{ref:neighbourhoodsofzero:prop} we have $x\in U$.
        It is enough to prove that $\ioneX^{-1}(U)\cap Z \to U$ is a locally closed immersion. In fact it is a
        closed immersion, as we prove below.

        Let $\Uplus$ represent $\Dfunctor{U}$.
        Since $x\in U$, we have $z\in \Uplus$, so
        $\Uplus \cap Z \neq\emptyset$.
        The morphism $\Uplus \to U$ is a
        closed immersion
        (Proposition~\ref{ref:representabilityForAffine:prop}) and
        trivially $\Uplus \cap Z\to \Uplus$ is a closed immersion, hence
        $\Uplus \cap Z\to \Uplus \to U$ is a closed immersion.
        The morphism $\Uplus \to \Xplus$ is an
        open immersion (Proposition~\ref{ref:openimmersions:prop}) and $\Uplus
        \subset \ioneX^{-1}(U)$, hence $\Uplus\to \ioneX^{-1}(U)$ is an open
        immersion. Thus,
        $\Uplus \cap Z \to \ioneX^{-1}(U)\cap Z$ is an open immersion.
        Summing up, we have the following Diagram~\eqref{eq:cancellation}.
        \begin{equation}\label{eq:cancellation}
            \begin{tikzcd}
                \Uplus \cap Z \arrow[r, "\mathrm{open}"]\arrow[d,
                "\mathrm{cl}"'] &\ioneX^{-1}(U)\cap Z\arrow[d]\\
                \Uplus\arrow[r, "\mathrm{cl}"'] & U
            \end{tikzcd}
        \end{equation}
        The map
        $\ioneX^{-1}(U)\cap Z\to U$ is separated as both these spaces are
        separated. Hence $\Uplus \cap Z\to \ioneX^{-1}(U)\cap Z$ is a closed immersion by
        cancellation. Then it is an open and closed immersion.
        The subset $\ioneX^{-1}(U)\cap Z \subset Z$ is
        an open subset of irreducible scheme, hence is irreducible.
        Thus, $\Uplus\cap Z\to \ioneX^{-1}(U) \cap Z$ is an isomorphism.
        Therefore, $\ioneX^{-1}(U) \cap Z\to U$ coincides with
        $\Uplus \cap Z \to U$ which is a closed immersion.
    \end{proof}
    \begin{remark}
        In contrast to Proposition~\ref{ref:ioneLocallyClosed:prop}, the question whether $\ioneX$ is
        an immersion on each \emph{connected} component is quite subtle,
        see~\cite[Appendix~B]{Drinfeld} for some counter-examples.
    \end{remark}

    We are now ready to prove the version of classical \BBname{}
    decomposition, as announced in the introduction.
    \begin{proof}[Proof of Theorem~\ref{ref:ABBdecomposition:thm:intro}]
        Smoothness of $\Xplus$ and the properties of $\iinftyX$ were proven in
        Corollary~\ref{ref:affineBundle:cor}. The multiplication by
        $0_{\Gbar}\in \Gbar$ is a morphism $r\colon\Xplus \to \Xplus$ which is
        retraction to $\isectionX(\varX^{\Group})$. Precisely, speaking, we
        have $r \circ \isectionX =\id_{\varX^{\Group}}$ and $r(\Xplus) =
        \isectionX(\varX^{\Group})$. Thus Point~\ref{item:ABBfirst} is
        proven. The remaining part of Point~\ref{item:ABBsecond} --- the
        properties of $\ioneX$ --- follow from
        Proposition~\ref{ref:ioneLocallyClosed:prop}, because the connected
        and irreducible components of every smooth variety coincide.
    \end{proof}

    \begin{remark}
        The classical \BBname{} decomposition has far reaching consequences
        for the homology of the underlying variety, see~\cite{Brosnan,
        Karpenko, Weber}. It would be interesting to see whether there are
        analogues taking into account the $\Gbar$-action on the cells.
    \end{remark}

    \begin{example}[Intersection of \BBname{} cells on the Hilbert scheme of
        plane]\label{ex:BBdecompositionForHilbA2}
        \def\Hilb{\mathcal{H}}%
        \def\Hilbplus{\Hilb^+}%
        \def\init#1{\operatorname{in}_{#1}}%
        \def\ordone{<_1}%
        \def\ordtwo{<_2}%
        Consider the Hilbert scheme $\Hilb$ of $d$ points on
        $\mathbb{A}^2=\Spec(S)$.
        As proven by Fogarty~\cite{fogarty}, it is a smooth, quasi-projective
        variety. For a monomial ordering $<$ there is an integral weight vector
        $w \in \mathbb{Z}^2$ such that $\init{<}(I) = \init{w}(I)$ for all $[I]\in
        \Hilb$. This weight vector induces an action
        \[
            \mu = \mu_{w}\colon \Gmult \times \Hilb \to \Hilb.
        \]
            Moreover, $\Hilb^{\Gmult}$ is the set of monomial ideals in
            $\Hilb$ and in particular it is finite.
        Hence the \BBname{} decomposition of $\Hilb$
        with respect to $\mu_{w}$ is a union of finitely many cells
        $Z_{[M]}^{<}$ isomorphic to affine spaces, labeled by monomial ideals
        $[M]\in \Hilb$, see~\cite{ellingsrud_stromme_1, ellingsrud_stromme_2}.
        Moreover, each of those cells embeds into $\Hilb$ as a locally closed
        subset. The above is classical.
        There is an ongoing work for concrete realization of these cells, for
        various classical orderings, see~\cite{CV08, Con11}.

        The aim of the present example is to prove that the intersection of
        two such cells corresponding to the same monomial ideal with
        respect to different orderings is isomorphic to an affine space.
        Fix two monomial orderings $\ordone$, $\ordtwo$. Fix corresponding weight
        vectors $w_1, w_2\in \mathbb{Z}^2$ and torus actions $\mu_1, \mu_2\colon \Gmult
        \times\Hilb \to \Hilb$.  Fix a monomial ideal $[M]\in \Hilb$ and let
        \[
            Z_1 = Z_{[M]}^{w_1}\subset \Hilb,\quad Z_2 = Z_{[M]}^{w_2} \subset \Hilb
        \]
        be the corresponding cells in the \BBname{} decompositions
        related to $w_1$, $w_2$. We will show that $Z_1
        \times_{\Hilb} Z_2 = Z_1 \cap Z_2$ is isomorphic to an affine space.
        Let $\mu = (\mu_1, \mu_2)\colon
        \Gmult^{\times 2} \times \Hilb \to \Hilb$. Note that
        \begin{equation}\label{eq:fixedPointsHilbert}
            \Hilb^{\Gmult, \mu_1} = \Hilb^{\Gmult, \mu_2} = \Hilb^{\Gmult
            \times \Gmult}
        \end{equation}
        as all three give the set of monomial ideals. Let $\mathbb{A}^2$ be
        a compacitification of $\Gmult^{\times 2}$ as in Example~\ref{ex:toric}.
        Let $\Hilbplus$ represent $\Dfunctor{\Hilb, \mathbb{A}^2}$. By
        Theorem~\ref{ref:ABBdecomposition:thm:intro}
        and by~\eqref{eq:fixedPointsHilbert}, the scheme $\Hilbplus$ is a
        disjoint union of affine spaces labeled by monomial ideals.
        Let $Z = Z_{[M]} \subset \Hilbplus$ be the cell corresponding to
        $[M]$. We claim that $Z  \simeq  Z_1 \times_{\Hilb} Z_2$ canonically.
        Indeed, the space $Z_1 \times_{\Hilb} Z_2$ represents the functor
        \[
            D(S) = \left\{ \varphi\colon \left(\mathbb{A}^2
            \setminus\{0\}\right) \times S\to \Hilb\
            |\ \varphi\mbox{ is }\Gmult\times \Gmult \mbox{ equivariant,
            }\varphi((1, 0) \times S) = \varphi((0, 1) \times S) = [M]\right\}
        \]
        whereas the space $Z$ represents the functor
        \[
            D'(S) = \left\{ \varphi'\colon \mathbb{A}^2 \times S\to \Hilb\
            |\ \varphi\mbox{ is }\Gmult\times \Gmult \mbox{equivariant,
            }\varphi'((0, 0) \times S) = [M]\right\}.
        \]
        The space $\Hilb$ is smooth, so by Remark~\ref{ref:locallyLinearSchemes:rmk}
        its point $[M]$ has an affine $(\Gmult\times\Gmult)$-stable
        neighbourhood $U$. Consider any family $\varphi\in D(S)$. Its equivariance
        and the conditions~$\varphi((*,0) \times S) = \varphi((0, *) \times S) = [M]$
        imply that $\varphi$ factors as $\left(\mathbb{A}^2
            \setminus\{0\}\right) \times S\to U \to \Hilb$, see
            Proposition~\ref{ref:neighbourhoodsofzero:prop}. Then
            the family $\varphi$ extends to $\varphi\colon
            \mathbb{A}^2 \times S\to U$ by the same argument as in
            Proposition~\ref{ref:compactificationImpliesMonoid:prop}.
            Therefore, $D(S) \subset D'(S)$. By~\eqref{eq:fixedPointsHilbert}
            we have $D'(S) \subset D(S)$. Thus $D = D'$ canonically and so
            $Z_1 \times_{\Hilb} Z_2 = Z$ is an affine space. Its dimension is
            the dimension of the positive orthant of
            $\mathbb{Z}^2$-graded tangent space to $\Hilb$, i.e.,
            $\dim Z = \dim_{\kk}\Hom(M, S/M)_{\geq 0,\geq 0}$.
            The same argument applies to show that the intersection of
            arbitrarily (but finitely) many Gr\"obner cells \emph{with the
			same initial ideal $M$} is isomorphic to an affine space. We do
            not know, whether a similar result holds when varying both the
            ordering and the initial ideal.
    \end{example}

    \section*{Acknowledgements}
    We thank Piotr Achinger, Jarod Alper, Andrzej \BBname{},  Michael Brion,
    Jaros{\l}aw Buczy{\'n}ski, Maksymilian Grab, David Rydh, Andrzej Weber, Torsten
    Wedhorn, and Jaros{\l}aw
    Wi{\'s}niewski for helpful comments.
    We also warmly thank the anonymous referee for careful reading and
    important comments, in particular pointing out missing assumptions on
    connectedness of $\Group$.
    We thank the caf{\'e}
    ``Po{\.z}egnanie z Afryk{\k{a}}'' in Bia{\l}ystok, where a significant
    part of the
    conceptual work was done, for the great atmosphere and kindness.

\end{document}